\documentclass[11pt]{article}

\usepackage{amsmath, amssymb, amsthm}
\usepackage{tikz}

\theoremstyle{plain}
\newtheorem{thm}{Theorem}[section]
\newtheorem*{thm*}{Theorem}
\newtheorem{prop}[thm]{Proposition}
\newtheorem{lem}[thm]{Lemma}
\newtheorem*{lem*}{Lemma}
\newtheorem{dfn}[thm]{Definition}

\newtheorem{conj}[thm]{Conjecture}

\newcommand{\BBE}{\mathbb{E}}
\newcommand{\BFP}{\mathbf{P}}

\date{}
\title{\vspace{-0.7cm}Ramsey numbers of degenerate graphs}

\author{
Choongbum Lee \thanks{Department of Mathematics,
MIT, Cambridge, MA 02139-4307. Email: cb\_lee@math.mit.edu.
Research supported by NSF Grant DMS-1362326.}
}

\begin{document}

\maketitle

\begin{abstract}
A graph is $d$-degenerate if all its subgraphs have a vertex of degree at most $d$.
We prove that there exists a constant $c$ such that 
for all natural numbers $d$ and $r$, 
every $d$-degenerate graph $H$ of chromatic number $r$ with 
$|V(H)| \ge 2^{d^22^{cr}}$ has Ramsey number at most $2^{d2^{cr}} |V(H)|$. 
This solves a conjecture of Burr and Erd\H{o}s from 1973.
\end{abstract}

\section{Introduction} \label{sec:intro}

Ramsey theory studies problems that can be grouped under
the common theme that `every large system contains a highly organized subsystem'.
A classical example is the celebrated van der Waerden theorem \cite{vdw} 
asserting that in every coloring of
the natural numbers with a finite number of colors, one can find 
monochromatic arithmetic progressions of arbitrary finite length.
This has motivated further results such as the Hales-Jewett theorem \cite{HaJe}
and Szemer\'edi's theorem \cite{Szemeredi} and has had a tremendous influence on 
Combinatorics and related fields.
See \cite{GrRoSp} for a comprehensive overview of Ramsey theory.

For a graph $H$, the {\em Ramsey number} of $H$, denoted $r(H)$, is defined as the minimum integer 
$n$ such that in every edge two-coloring of $K_n$, the complete graph on $n$ vertices,
there exists a monochromatic copy of $H$.
The name of the field has its origins in a 1930 paper of Frank P. Ramsey \cite{Ramsey},
who proved that $r(K_t)$ are finite for all natural numbers $t$.
In 1935, Erd\H{o}s and Szekeres \cite{ErSz} rediscovered Ramsey's theorem,
and brought it to more widespread attention by 
finding interesting geometric applications of its hypergraph extensions. 
Plenty of variants and applications have been found since then, 
and now it is considered to be one of the most important results in combinatorics,
lying at the point of intersection of many fields.

There are many fascinating problems concerning bounds on Ramsey numbers of various graphs.
Erd\H{o}s and Szekeres, in the paper mentioned above, established a recurrence relation 
on the Ramsey numbers of complete graphs that implies
$r(K_t) \le {2t-2 \choose t-1} = 2^{(2+o(1))t}$ for all natural numbers $t$.
Later, in 1947, Erd\H{o}s \cite{Erdos}, 
in one of the earliest applications of the probabilistic method, proved 
$r(K_t) \ge 2^{(1/2+o(1))t}$. 
These two bounds together show that $r(K_t)$ is exponential in terms of the number of vertices $t$. 
There have been some interesting improvements on these 
bounds \cite{Conlon_diag, Spencer},
but despite a great amount of effort, 
the constants in the exponents remain unchanged. See the recent survey paper of 
Conlon, Fox, and Sudakov \cite{CoFoSu_survey} for further information on graph Ramsey theory.

In 1973, Burr and Erd\H{o}s \cite{BuEr} initiated the study of Ramsey numbers of sparse graphs and conjectured
that the behavior of Ramsey numbers of sparse graphs should be dramatically different from that
of complete graphs.
A graph is {\em $d$-degenerate} if all its subgraphs contain a vertex of degree at most $d$. 
Degeneracy is a natural measure of sparseness of graphs as it implies
that for all subsets of vertices $X$, there are fewer than
$d|X|$ edges with both endpoints in $X$.
Burr and Erd\H{o}s conjectured that for every natural number $d$,
there exists a constant $c = c(d)$ such that every $d$-degenerate graph $H$ on $n$ vertices satisfies
$r(H) \le c n$. This is in striking contrast with the case of complete graphs
where the dependence on the number of vertices is exponential.

This conjecture has received much attention and 
motivated several important developments over the past 40 years. 
For example, the work of Chv\'atal, R\"odl, Szemer\'edi, and Trotter \cite{ChRoSzTr} from 1983
that we will discuss below is one of the earliest 
applications of the regularity lemma, and fostered further developments such
as the blow-up lemma of Koml\'os, Sark\"ozy, and Szemer\'edi \cite{KoSaSz}.
Furthermore, progress on the conjecture is intertwined with the development
of a powerful new tool in probabilistic combinatorics now known as 
dependent random choice.
Kostochka and R\"odl \cite{KoRo01} 
used its primitive version to study a special case of the conjecture, 
and later, using a different method \cite{KoRo04}, established the 
first polynomial bound $r(H) \le c_d n^2$ for all $d$-degenerate graphs $H$
on $n$ vertices (where $c_d$ is a constant depending on $d$).
The general framework of applying the dependent random choice technique to embed degenerate graphs 
was pioneered by Kostochka and Sudakov \cite{KoSu}, who improved the bound of
Kostochka and R\"odl to a nearly linear bound
$r(H) \le 2^{c_d (\log n)^{2d/(2d+1)}}n$.
Later, Fox and Sudakov \cite{FoSu09, FoSu09-2} refined the method to prove 
that $r(H) \le 2^{c_d \sqrt{\log n}}n$. 
See the survey paper of Fox and Sudakov \cite{FoSu11}
for an overview of the history and applications of this fascinating method.

Linear bounds were established for special cases such as
subdivisions of graphs by Alon \cite{Alon}, random graphs by Fox and Sudakov \cite{FoSu09-2}, and
graphs with small bandwidth by the author \cite{Lee}.
The conjecture has also been examined 
for weaker notions of sparseness. Chv\'atal, R\"odl, Szemer\'edi, and Trotter \cite{ChRoSzTr}
proved that if $H$ is a graph on $n$ vertices of maximum degree at most $\Delta$, 
then $r(H) \le c(\Delta) n$, where $c(\Delta)$ is a constant depending only on $\Delta$. 
Their proof used the regularity lemma and thus the dependence of $c(\Delta)$ on $\Delta$ 
was of tower type. This bound has been improved since then, by Eaton \cite{Eaton}, 
Graham, R\"odl, and Ruci\'nski \cite{GrRoRu00, GrRoRu01}, and then by 
Conlon, Fox, and Sudakov \cite{CoFoSu} to $r(H) \le c^{\Delta \log \Delta} n$.
For bipartite graphs, Conlon \cite{Conlon} and independently Fox and Sudakov \cite{FoSu09}
proved that a stronger bound $r(H) \le c^{\Delta} n$ holds.
These results are close to best possible since Graham, R\"odl, and Ruci\'nski \cite{GrRoRu01} 
constructed bipartite graphs $H$ on $n$ vertices of maximum degree $\Delta$ satisfying 
$r(H) \ge c^{\Delta}n$ (for a different constant $c$).

Chen and Schelp \cite{ChSc} considered another measure of sparseness.
They defined a graph to be {\em $p$-arrangeable} if there is an ordering $v_1, \cdots, v_n$
of the vertices such that for any vertex $v_i$, its neighbors to the right of $v_i$ 
have together at most $p$ neighbors to the left of $v_i$ (including $v_i$). 
This is a measure of sparseness that lies strictly between degeneracy and bounded
maximum degree. They showed that graphs with bounded arrangeability have Ramsey number
linear in the number of vertices, implying, in particular, that the
Burr-Erd\H{o}s conjecture holds for planar graphs. Furthermore, 
R\"odl and Thomas \cite{RoTh} showed that graphs with no $K_p$-subdivision 
have arrangeability less than $p^8$, and therefore have Ramsey number
linear in the number of vertices.

\medskip

In this paper, we build upon these developments and
settle the conjecture of Burr and Erd\H{o}s.
We say that a graph $G$ is {\em universal} for a family $\mathcal{F}$ of graphs if
it contains all graphs $F \in \mathcal{F}$ as subgraphs.
For an edge coloring of a graph, we say that a color is {\em universal} for a family $\mathcal{F}$
if the subgraph induced by the edges of that color is universal for $\mathcal{F}$.

\begin{thm} \label{thm:main}
There exists a constant $c$ such that the following holds for every natural number $d$, $r$, and $n$ satisfying $n \ge 2^{d^2 2^{cr}}$.
For every edge two-coloring of the complete graph on at least $2^{d2^{cr}}n$ vertices, one of 
the colors is universal for the family of 
$d$-degenerate $r$-colorable graphs on at most $n$ vertices.
\end{thm}

This settles the conjecture of Burr and Erd\H{o}s since all $d$-degenerate graphs
have chromatic number at most $d+1$. 
Moreover, for fixed values of $r$,
Theorem~\ref{thm:main} is best possible up to the constant in the exponent. 
To see this, consider a random graph $G$ on $(1-\varepsilon)2^{d}n$ vertices of density $\frac{1}{2}$
and let $H$ be the complete bipartite graph $K_{d,n-d}$ with $d$ vertices in one part
and $n-d$ vertices in the other part. It is well-known that in $G$, with high probability, 
every $d$-tuple of (distinct) vertices has fewer than $(1-\frac{\varepsilon}{2})n$ common neighbors. 
Therefore, $G$ does not contain a copy of $H$. 
Since the complement of $G$ can also be seen as a random graph of density $\frac{1}{2}$, 
we see that the complement of $G$ also does not contain a copy of $H$. 
Another way to see the tightness of Theorem~\ref{thm:main} is by considering
the construction of Graham, R\"odl, and Ruci\'nski \cite{GrRoRu01} mentioned above.

The {\em density} of a graph is defined as the fraction of pairs of distinct vertices that form an edge. 
Most of the previous results mentioned above in fact provide a density-embedding result 
for bipartite graphs, saying that every dense enough graph contains a copy.
Note that the Ramsey number result follows from such a density-embedding result
since in every edge two-coloring of a complete graph,
one of the colors must have density at least $\frac{1}{2}$.
In this context, the following theorem generalizes Theorem~\ref{thm:main} to a
density-embedding result.

\begin{thm} \label{thm:bipartite}
There exists a constant $c$ such that the following holds for every natural number $d$ and real number $\alpha \le \frac{1}{2}$.
For every natural number $n \ge \alpha^{-cd^2}$, if $G$ is a graph with at least $\alpha^{-cd} n$ vertices and density at least $\alpha$, 
then it is universal for the family of $d$-degenerate bipartite graphs on $n$ vertices.
\end{thm}

Note that the complete bipartite graph $K_{d,n-d}$ mentioned above has a specific structure.
Namely, every vertex on one side has bounded degree. Until now, even this special case,
bipartite graphs with one side of bounded degree, 
of the conjecture was open.
A result of Alon \cite{Alon} implies the case when the degrees are bounded by two.
The corresponding density-embedding result was proved by Fox and Sudakov \cite{FoSu09},
who suggested the general case as an interesting problem to examine. 
The theorem below addresses this special case of the conjecture and shows that 
we can improve Theorem~\ref{thm:bipartite} to get nearly best possible embedding results. 

\begin{thm} \label{thm:osb}
Let $d$ be a natural number and $\alpha, \varepsilon$ be positive real numbers 
satisfying $\alpha^{d(d-2)} \le \varepsilon < 1$. 
Let $G$ be a graph on $(1+\varepsilon)\alpha^{-d} n$ vertices of density at least $\alpha$.
Then $G$ is universal for the family of bipartite graphs $H$ on $n$ vertices with a 
vertex partition $W_1 \cup W_2$ where all vertices in $W_1$ have at most $d$ neighbors in $W_2$,
and $\frac{|W_2|^d}{|W_2|(|W_2|-1) \cdots (|W_2|-d+1)} \le 1 + \varepsilon$.
\end{thm}

This theorem strengthens previous density-embedding results for bipartite graphs
with bounded maximum degree \cite{Conlon, FoSu09}, and
is close to being best possible as can be seen by the example where $G$ is a random graph
and $H$ is the complete bipartite graph $K_{d,n-d}$ discussed above. 
Let $Q_n$ be the {\em hypercube}, which is the graph 
with vertex set $\{0,1\}^{n}$ where two vertices are adjacent
if and only if they differ in exactly one coordinate. 
Theorem~\ref{thm:osb} with $\varepsilon = \frac{n^2}{2^n}$ and $\alpha = \frac{1}{2}$ shows that
$r(Q_n) \le 2^{2n} + n^2 2^n$ holds for all sufficiently large $n$.
This bound improves by a constant factor, 
the current best known bound $r(Q_n) \le 2^{2n+6}$
of Conlon, Fox, and Sudakov \cite{CoFoSu_ext}, proved using a different approach 
based on the local lemma.
It is conjectured \cite{BuEr} that there exists a constant $c$ such that 
$r(Q_n) \le c 2^n$ for all $n$. 

The proofs of the three theorems above are based on dependent random choice 
and build upon several ideas developed for earlier applications.
While the proofs of Theorems~\ref{thm:main} and \ref{thm:bipartite} are technically involved,
Theorem~\ref{thm:osb} has a short proof that highlights one of the main 
differences in our usage of this technique as compared to the previous ones.
We thus start by presenting the proof of Theorem~\ref{thm:osb} in Section~\ref{sec:one_side_bounded}. 
In Section~\ref{sec:preliminaries}, we introduce some central concepts and notation and give a 
brief outline of the proofs of our main theorems. 
In Section~\ref{sec:embedding}, we develop the embedding strategy 
that will be used. The proofs of Theorems~\ref{thm:main} and \ref{thm:bipartite} will be
given in Section~\ref{sec:proof_main} except for one key lemma, which will be proved in 
Section~\ref{sec:pruning}.
We conclude with some remarks in Section~\ref{sec:remarks}.

\medskip

\noindent \textbf{Notation}. 
For an integer $m$, define $[m] := \{1,2,\ldots, m\}$ and for a pair of integers $m_1, m_2$,
define $[m_1, m_2] := \{m_1, \cdots, m_2\}$, $[m_1, m_2) := \{m_1, \cdots, m_2-1\}$,
$(m_1, m_2] := \{m_1 + 1, \cdots, m_2\}$, and $(m_1, m_2) := \{m_1+1, \cdots, m_2-1\}$.
For a set of elements $X$, we define $X^t = X \times \cdots \times X$ as the set of all $t$-tuples in $X$.
Let $Q$ be a $d$-tuple and $Q'$ be a $d'$-tuple of elements in some set.
We use $Q \cup Q'$ to denote the $(d+d')$-tuple obtained by concatenating $Q'$ to the 
end of $Q$.  

Let $G=(V,E)$ be an $n$-vertex graph.
For a vertex $x$ and a set $T$, we define $\deg(x; T)$ as the number of neighbors of $x$ in $T$, 
and define $\deg(x) := \deg(x; V)$.
For a set or an ordered tuple of vertices $Q$, define 
$N(Q;T) := \{ x \in T : \{x,y\} \in E,\,\forall y \in Q \}$ as the set of common neighbors in $T$ 
of vertices in $Q$. Define $N(Q) := N(Q; V)$.
For two sets $X$ and $Y$, define $E(X,Y) = \{(x,y) \in X \times Y \,:\, \{x,y\} \in E\}$
and $e(X,Y) = |E(X,Y)|$. Furthermore, define $e(X) = \frac{1}{2}e(X,X)$ as the number of 
edges in $X$. Let $d(X,Y) = \frac{e(X,Y)}{|X||Y|}$ be the density of edges between
$X$ and $Y$.
For a graph $H$, an {\em embedding} of $H$ in $G$ is an injective map $f: V(H) \rightarrow V(G)$
for which $\{f(v), f(w)\} \in E(G)$ whenever $\{v,w\} \in E(H)$.
A {\em partial embedding of $H$ in $G$ defined on $V' \subseteq V(H)$} is an embedding of
$H[V']$ into $G$. For $V'' \subseteq V(H)$ an {\it extension to $V''$} of a partial 
embedding $f$ on $V'$ is an embedding $g : H[V' \cup V''] \rightarrow G$ such that $g|_{V'} = f$.
We often abuse notation and denote the extended map using the same notation $f$.

Throughout the paper, we will be using subscripts such as in $x_{\ref{thm:main}}$ to indicate 
that $x$ is the constant coming from Theorem/Lemma/Proposition 1.1.

\section{One-side bounded bipartite graphs} \label{sec:one_side_bounded}

Fox and Sudakov \cite{FoSu09} asked if the Burr-Erd\H{o}s conjecture holds for bipartite graphs where 
all vertices in one part have their degree bounded by $d$.
In this section, we answer their question by
establishing an embedding theorem for such graphs. 

\begin{dfn} \label{dfn:defect_simple}
Let $G$ be a bipartite graph with vertex partition $V_1 \cup V_2$.
For a positive real number $\theta$ 
and an ordered $d$-tuple $Q \in V_2^d$, we define the {\em $\theta$-defect of $Q$} as
\[
	\omega_{\theta}(Q) = 
	\begin{cases}
	 0 	& \textrm{ if } |N(Q)| \ge \theta \\
	 \frac{\theta}{|N(Q)|} & \textrm{otherwise}.
	\end{cases}
\]
We simply write $\omega(Q)$ when $\theta$ is clear from the context.
\end{dfn}

For a given ordered $d$-tuple of vertices $Q$,
the $\theta$-defect $\omega_\theta(Q)$ defined above captures how the number of common neighbors of $Q$
compares to some prescribed threshold $\theta$. 
We give a penalty to the $d$-tuples having only a small number of common neighbors.
Informally, if a set has small average $\theta$-defect over $d$-tuples, then
most $d$-tuples in it will have at least $\theta$ common neighbors.
This weight function has been considered before by Alon, Krivelevich, and Sudakov \cite{AlKrSu} 
with a fixed value of the threshold $\theta$. 
The following lemma, based on dependent random choice, shows that we can find a set in which the 
average defect is small.

\begin{lem} \label{lem:drc_simple}
Let $d, t$ be natural numbers, $s$ be a non-negative integer satisfying $t \ge s$, and
$\eta, \varepsilon$ be positive real numbers satisfying $\varepsilon < 1$.
Let $G$ be a bipartite graph of density at least $\alpha$ with vertex 
partition $V_1 \cup V_2$. 
Then there exists a set $A \subseteq V_2$ of size at least $|A| \ge \varepsilon^{1/d} \alpha^{t}|V_2|$
such that $\frac{1}{|A|^d}\sum_{Q \in A^d} \omega_{\theta}(Q)^{s} \le \frac{\eta^{t}}{1-\varepsilon}$ 
for all $\theta \le \eta \alpha^{d}|V_1|$.
\end{lem}
\begin{proof}
Since $\omega_{\theta}$ is increasing in $\theta$, it suffices to prove the lemma
when $\theta = \eta\alpha^{d}|V_1|$. 
Throughout the proof we will use this fixed value of $\theta$, and use
the notation $\omega_\theta$ without the subscript. 
Let ${\bf X} \in V_1^t$ be a $t$-tuple of vertices chosen uniformly at random 
(where we allow repetition of vertices), and define
${\bf A} = N({\bf X})$. Note that
\[
	\BBE\left[|{\bf A}|\right]
	= \sum_{v \in V_2} \BFP(v \in N({\bf X}))
	= \sum_{v \in V_2} \left( \frac{\deg(v)}{|V_1|} \right)^t
	\ge |V_2| \left( \frac{1}{|V_2|}\sum_{v \in V_2}  \frac{\deg(v)}{|V_1|} \right)^t
	\ge \alpha^t |V_2|.
\]
By convexity, we have $\BBE\left[|{\bf A}|^d\right] \ge \alpha^{dt} |V_2|^d$.
Fix a $d$-tuple $Q \in V_2^d$ and let $|N(Q)| = \gamma \theta$. Since
\[
	\BFP\left( Q \in {\bf A}^{d} \right)
	= \left( \frac{|N(Q)|}{|V_1|} \right)^{t}
	= \left( \frac{\gamma \eta \alpha^d|V_1|}{|V_1|} \right)^{t}
	= \gamma^t \eta^t \alpha^{dt},
\]
if $\gamma < 1$, then $\omega(Q)^s \cdot \BFP\left( Q \in {\bf A}^{d} \right) = \gamma^{t-s} \eta^{t}\alpha^{dt} \le \eta^{t}\alpha^{dt}$ (we used the fact that $t \ge s$). 
On the other hand, if $\gamma \ge 1$, then $\omega(Q) = 0$ by definition.
Hence,
\[
	\BBE\left[\sum_{Q \in {\bf A}^{d}} \omega(Q)^s\right] 
	= \sum_{Q \in V_2^d} \omega(Q)^s \cdot \BFP\left( Q \in {\bf A}^{d} \right)
	\le |V_2|^d \cdot \eta^{t} \alpha^{dt}.
\]
Therefore,
\[
	\BBE\left[|{\bf A}|^d - \frac{1-\varepsilon}{\eta^t} \sum_{Q \in {\bf A}^{d}} \omega(Q)^s\right]
	\ge \varepsilon \alpha^{dt} |V_2|^d.
\]
Choose ${\bf X}$ so that the random variable on the left-hand side
of the inequality above becomes at least as large as its expectation and let $A$ be
${\bf A}$ for this choice of vertices.  
Then $|A|^d \ge \varepsilon \alpha^{dt}|V_2|^d$, and thus $|A| \ge \varepsilon^{1/d} \alpha^{t}|V_2|$ .
The other claim follows from 
$|A|^d - \frac{1-\varepsilon}{\eta^t}\sum_{Q \in A^d} \omega(Q)^s \ge 0$.
\end{proof}

We now prove a density-embedding theorem for bipartite graphs with one part having
bounded maximum degree. Theorem~\ref{thm:osb}  follows from the theorem below by taking
$\varepsilon \ge \alpha^{d(d-2)}$ since it implies 
$(1+\varepsilon)\alpha^{-d} \ge (\frac{1+\varepsilon}{\varepsilon})^{1/d}\alpha^{-2}$, and in every graph on $n = n_1 + n_2$ vertices of density $\alpha$,
we can find a bipartite subgraph with $n_1$ and $n_2$ vertices in each part, having density 
at least $\alpha$. 

\begin{thm} \label{thm:one_side_bounded}
Let $\varepsilon$ be a positive real number.
Let $H$ be a bipartite graph on $n$ vertices with a 
vertex partition $W_1 \cup W_2$ where all vertices in $W_1$ have at most $d$ neighbors in $W_2$. 
Let $G$ be a bipartite graph of density at least $\alpha$ with vertex 
partition $V_1 \cup V_2$ where $|V_1| \ge (1+\varepsilon)\alpha^{-d} |W_1|$ and $|V_2| \ge (\frac{1+\varepsilon}{\varepsilon})^{1/d}\alpha^{-2}|W_2|$. 
If $\frac{|W_2|^d}{|W_2|(|W_2|-1)\cdots(|W_2|-d+1)} \le 1 + \varepsilon$, 
then $G$ contains a copy of $H$.
\end{thm}
\begin{proof}
Let $\theta = |W_1|$ and note that $\theta \le \frac{1}{1+\varepsilon}\alpha^{d}|V_1|$.
Throughout the proof we will use $\omega$ without subscript with the understanding 
that $\omega = \omega_\theta$.
Apply Lemma~\ref{lem:drc_simple} with $t_{\ref{lem:drc_simple}} = 2$, $s_{\ref{lem:drc_simple}} = 1$,
$d_{\ref{lem:drc_simple}} = d$,
$\eta_{\ref{lem:drc_simple}} = \frac{1}{1+\varepsilon}$, and $\varepsilon_{\ref{lem:drc_simple}} = \frac{\varepsilon}{1+\varepsilon}$
to find a set $A \subseteq V_2$ of size 
$|A| \ge (\frac{\varepsilon}{1+\varepsilon})^{1/d}\alpha^2|V_2| \ge |W_2|$ for which 
$\sum_{Q \in A^d} \omega(Q) \le (\frac{1}{1+\varepsilon})^2 \frac{1}{1-\varepsilon/(1+\varepsilon)}|A|^d
= \frac{1}{1+\varepsilon}|A|^d$.
By adding edges if necessary, we may assume that all vertices in $W_1$ have
exactly $d$ neighbors in $W_2$. 
Let $\phi$ be an injective map from $W_2$ to $A$ chosen uniformly at random.
For each vertex $v \in W_1$, let $e_v$ be an
ordered $d$-tuple of vertices obtained from $N(v)$ by arbitrarily ordering the vertices. 
Note that
\begin{align*}
	\BBE\left[ \sum_{v \in W_1} \omega(\phi(e_v)) \right]
	= \sum_{v \in W_1} \BBE\left[ \omega(\phi(e_v)) \right]
	=&\, \sum_{v \in W_1} \sum_{Q \in A^d}\omega(Q) \BFP(\phi(e_v)=Q) \\
	\le&\, |W_1| \cdot  \sum_{Q \in A^d} \frac{\omega(Q)}{|A|(|A|-1) \cdots (|A|-d+1)}.
\end{align*}
Since $\sum_{Q \in A^d} \omega(Q) \le \frac{1}{1+\varepsilon}|A|^d$ and $|A|\ge |W_2|$,
\[
	\BBE\left[ \sum_{v \in W_1} \omega(\phi(e_v)) \right] 
	\le |W_1| \cdot \frac{1}{1+\varepsilon}\frac{|W_2|^d}{|W_2|(|W_2|-1)\cdots(|W_2|-d+1)}\le |W_1|.
\]	
Therefore, there exists a particular choice of $\phi$ such that $\sum_{v \in W_1} \omega(\phi(e_v)) \le |W_1|$. 
Note that $\phi$ is trivially a partial embedding of $H$ in $G$ defined on $W_2$. 
We will now extend $\phi$ to $W_1$. 
Order the vertices in $W_1$ as $v_1, \cdots, v_{n_1}$
in decreasing order of $\omega(\phi(e_{v_i}))$ (where ties are broken arbitrarily).
We will greedily embed the vertices of $W_1$ following this order.
Suppose that we extended $\phi$ to $\{v_1, \cdots, v_{i-1}\}$, and for simplicity define $e_i = e_{v_i}$. 
First assume that $\omega(\phi(e_{i})) \neq 0$. 
Note that $i \cdot \omega(\phi(e_{i})) \le \sum_{v \in W_1} \omega(\phi(e_v)) \le |W_1|$ and therefore $\omega(\phi(e_{i})) \le \frac{|W_1|}{i}$. 
Hence,
\[
	|N(\phi(e_{i}))| = \frac{\theta}{\omega(\phi(e_{i}))} \ge |W_1| \cdot \frac{i}{|W_1|} = i. 
\]
Since we have so far embedded at most $i-1$ vertices of $W_1$, we can define $\phi(v_i)$
as a vertex in $N(\phi(e_{i})) \setminus \{\phi(v_1), \cdots, \phi(v_{i-1})\}$. 
On the other hand, if $\omega(\phi(e_{i})) = 0$, then $|N(\phi(e_{i}))| \ge \theta \ge |W_1|$
and therefore we trivially have $N(\phi(e_{i})) \setminus \{\phi(v_1), \cdots, \phi(v_{i-1})\} \neq \emptyset$
and can define $\phi(v_i)$ as a vertex in this set. Thus, we can find an embedding of $H$ in $G$.
\end{proof}

\section{Preliminaries} \label{sec:preliminaries}

\subsection{Decomposing $H$}

We start with the following simple lemma that decomposes a given degenerate graph into manageable pieces.

\begin{lem} \label{lem:split}
Let $H$ be an $r$-colorable $d$-degenerate graph on $n$ vertices. 
Then there exist a natural number $k$ and disjoint subsets $\{W_i^{(j)}\}_{i \in [k], j \in [r]}$
with the following properties:
\begin{itemize}
  \setlength{\itemsep}{1pt} \setlength{\parskip}{0pt}
  \setlength{\parsep}{0pt}
\item[(i)] $k \le \log_{2} n$,
\item[(ii)] for all $(i,j) \in [k] \times [r]$, $|W_{i}^{(j)}| \le 2^{-i+1}n$,
\item[(iii)] for all $j \in [r]$, the set $\bigcup_{i \in [k]} W_{i}^{(j)}$ is an independent set, and 
\item[(iv)] for all $(i,j) \in [k] \times [r]$, each vertex in $W_i^{(j)}$ has at most $4d$ neighbors in
$\bigcup_{i'\ge i, j' \in [r]} W_{i'}^{(j')}$.
\end{itemize}
\end{lem}
\begin{proof}
Let $U_1= V(H)$. For each $i \ge 1$, define $U_{i+1} \subseteq U_i$ as the set of vertices
having degree at least $4d$ in the subgraph $H[U_i]$. 
Since $H$ is $d$-degenerate, there are at most $d|U_i|$ edges in the subgraph $H[U_i]$. 
Therefore, we have $4d |U_{i+1}| \le 2d|U_i|$, forcing $|U_{i+1}| \le \frac{1}{2}|U_i|$.
Since $H$ has $n$ vertices, we have $|U_i| \le 2^{-(i-1)}n$ for each $i\ge1$, 
and the process continues for $k \le \log_{2} n$ steps.
Define $W_i = U_i \setminus U_{i+1}$ for all $i \in [k]$. 
Consider a proper $r$-coloring of $H$ using $[r]$, and for each $j \in [r]$, 
define $W_i^{(j)}$ as the set of vertices of color $j$ in $W_i$. 
One can easily check that all the conditions hold. 
\end{proof}

\subsection{Defect and average defect}

We will work with the decomposition given by Lemma~\ref{lem:split}. 
The main objective is to find a partition $\{V_{i}^{(j)}\}$ of the vertex set of the host graph $G$
so that we can embed $W_{i}^{(j)}$ to $V_{i}^{(j)}$ piece by piece. 
As in the previous section, the existence of such an embedding depends on the `average defect'
of this partition. In this subsection, we introduce the central concepts and notation
related to the defect function. 
We slightly generalize Definition~\ref{dfn:defect_simple} so that we consider 
the number of common neighbors in a specific set.

\begin{dfn} \label{dfn:defect}
Let $G$ be a graph.
For a positive real number $\theta$, a set $T \subseteq V(G)$,
and an ordered tuple $Q$ of vertices, we define the {\em $\theta$-defect of $Q$ in $T$} as
\[
	\omega_{\theta}(Q; T) = 
	\begin{cases}
	 0 	& \textrm{ if } |N(Q; T)| \ge \theta \\
	\frac{\theta}{|N(Q;T)|} & \textrm{otherwise}.
	\end{cases}
\]
We may simply write $\omega(Q; T)$ when $\theta$ is clear from the context. 
\end{dfn}

One can easily check that monotonicity $\omega_{\theta}(Q; T) \le \omega_{\theta'}(Q'; T')$ holds for all
$\theta \le \theta'$, $Q \subseteq Q'$ (as sets), and $T \supseteq T'$. 
Note that in the previous subsection, we utilized Lemma~\ref{lem:drc_simple} only
for the case $s=1$, even though Lemma~\ref{lem:drc_simple} provided a variety of bounds.
From now on, it will be crucial to consider other values of $s$.

\begin{dfn}
Let $G$ be a graph and $d, s$ be natural numbers. 
For a positive real number $\theta$, a set $T \subseteq V(G)$, and a set $\mathcal{Q} \subseteq V(G)^d$,
we define the {\em $s$-th moment of the $\theta$-defect of $\mathcal{Q}$ in $T$} as
\[
	\mu_{s,\theta}(\mathcal{Q}; T) = \frac{1}{|\mathcal{Q}|} \sum_{Q \in \mathcal{Q}} \omega_{\theta}(Q; T)^s.
\]
We may simply write $\mu_s(\mathcal{Q}; T)$ when $\theta$ is clear from the context. 
We extend the definition to $s=0$ by defining $\omega_{\theta}(Q;T)^0$ to be 0 when $\omega_{\theta}(Q;T)=0$, and 1 otherwise.
\end{dfn}

The monotonicity of the defect function $\omega$ and the fact that it equals either 0 or
a real number at least 1 implies
$\mu_{s,\theta}(\mathcal{Q}; T) \le \mu_{s',\theta'}(\mathcal{Q}; T')$ for
$s \le s'$, $\theta \le \theta'$, and $T \supseteq T'$.
Note that the conclusion of Lemma~\ref{lem:drc_simple} can be re-written as
an upper bound on $\mu_{s,\theta}(A^{d}; V_1)$.
Most previous applications of dependent random choice were proved by controlling
the $0$-th moment of the $\theta$-defect function, which equals
the probability that a random uniform $d$-tuple in $\mathcal{Q}$ has
fewer than $\theta$ common neighbors, whereas here we will be considering higher moments.
We will focus on the cases when $\mathcal{Q} = A_1 \times \cdots \times A_d$ for some 
(not necessarily distinct) sets $A_i$. The following proposition provides a useful
relation between average defects of such product sets.

\begin{prop} \label{prop:productsets}
If $A_1, \cdots, A_d, T$ are vertex subsets,
then $\mu_{s,\theta}(\prod_{i=1}^{d-1}A_i; T) \le \mu_{s,\theta}(\prod_{i=1}^{d}A_i; T)$.
\end{prop}
\begin{proof}
Since $\omega_{\theta}(Q; T) \le \omega_{\theta}(Q'; T)$ for all $Q \subseteq Q'$,
\begin{align*}
	\sum_{Q \in \prod_{i=1}^{d-1} A_{i}} \omega_{\theta}(Q; T)^s
	\le &\,
	\sum_{Q \in \prod_{i=1}^{d-1} A_{i}} \left(\frac{1}{|A_{d}|}\sum_{v_d \in A_{d}} \omega_{\theta}(Q \cup \{v_d\}; T)^s\right)
	= \frac{1}{|A_{d}|} \sum_{Q \in \prod_{i=1}^{d} A_i} \omega_{\theta}(Q; T)^s.
\end{align*}
Divide both sides of the inequality by $\prod_{i=1}^{d-1}|A_i|$ to obtain $\mu_{s,\theta}(\prod_{i=1}^{d-1}A_i; T) \le \mu_{s,\theta}(\prod_{i=1}^{d}A_i;T)$.
\end{proof}

The following proposition shows that the contribution of $d$-tuples with
repeated vertices towards the average defect is of small magnitude.

\begin{prop} \label{prop:boundary}
Let $A_1, \cdots, A_d$ be vertex subsets all of sizes at least $m$,
and let $T$ be a vertex subset. 
Let $\partial \mathcal{Q} \subseteq \prod_{i \in [d]} A_i$ 
be the set of $d$-tuples $(v_1, \cdots, v_d)$ for which $v_i = v_j$ for some distinct $i,j \in [d]$. 
Then $\sum_{Q \in \partial\mathcal{Q}} \omega_{\theta}(Q;T)^s \le \frac{d(d-1)}{2m} \sum_{Q \in \prod_{i=1}^{d} A_{i}} \omega_{\theta}(Q; T)^s$.
\end{prop}
\begin{proof}
For each distinct $a,b \in [d]$, define $\mathcal{Q}_{a,b}$ as the 
set of $d$-tuples in $\partial \mathcal{Q}$ whose $a$-th and $b$-th coordinates coincide.
Note that if $Q = (v_1, \cdots, v_d)$
and $v_{d-1} = v_d$, then $\omega_{\theta}(Q;T) = \omega_{\theta}((v_1, \cdots, v_{d-1}); T)$. 
Therefore, by Proposition~\ref{prop:productsets},
\[
	\sum_{Q \in \mathcal{Q}_{d-1,d}} \omega_{\theta}(Q;T)^s 
	= \sum_{Q \in \prod_{i=1}^{d-1}A_i} \omega_{\theta}(Q;T)^s 
	\le \frac{1}{|A_d|} \sum_{Q \in \prod_{i=1}^{d} A_{i}} \omega_{\theta}(Q; T)^s.
\]
Since $|A_i| \ge m$ for all $i \in [d]$, the conclusion follows by summing the inequalities over all pairs $a,b$. 
\end{proof}

The next proposition asserts that for a pair of sets $V_1$ and $V_2$, 
if the $s$-th moment of the average defect of $V_1^d$ into $V_2$ is small, then 
all $d$-tuples in $V_1^d$ have many common neighbors in $V_2$.

\begin{prop} \label{prop:mindeg}
Let $d, s$, and $\theta$ be natural numbers. 
For all sets of vertices $V_1$ and $V_2$, the number of
$d$-tuples $Q \in V_1^d$ satisfying $|N(Q; V_2)| < \frac{\theta}{|V_1|^{d/s}}$
is less than $\mu_{s,\theta}(V_1^d; V_2)$. 
\end{prop}
\begin{proof}
If a $d$-tuple $Q \in V_{1}^d$ satisfies $|N(Q; V_2)| < \frac{\theta}{|V_1|^{d/s}}$,
then by definition, we have $\omega_{\theta}(Q; V_2)^s > |V_1|^{d}$.
Therefore, the number of such $d$-tuples is less than $\mu_{s,\theta}(V_1^d; V_2)$. 
\end{proof}

We will use Proposition~\ref{prop:mindeg} mostly in the case when $\mu_{s,\theta}(V_1^d; V_2) < 1$. 
For such cases, all $d$-tuples $Q \in V_1^d$ satisfy $|N(Q; V_2)| \ge \frac{\theta}{|V_1|^{d/s}}$.

\subsection{Outline of proof}

Let $H$ be a $d$-degenerate graph and $G$ be a host graph that we would like to embed $H$ into.
Let $\{W_i\}_{i \in [k]}$ be the partition of the vertex set of $H$ obtained from Lemma~\ref{lem:split}
and suppose that we have disjoint vertex subsets $\{V_i\}_{i \in [k]}$ of $G$.
We will embed $W_i$ in $V_i$ one vertex at a time, starting at $i =k$ and 
proceeding in decreasing order of index. 
Suppose that we have successfully embedded $W_{i+1}$ in $V_{i+1}$ (and all previous steps) and are
about to embed $W_i$ in $V_i$. Let $\phi$ denote the partial embedding that we have at that moment. 
For each vertex $v \in W_i$, define $N^+(v) = N(v) \cap \bigcup_{j > i} W_j$.
For simplicity, we assume that $|N^+(v)| = d$ for all $v \in V(H)$.
We can force the extension of $\phi$ to $W_i$ to be a homomorphism
by embedding each vertex $v \in W_i$ to a common neighbor of vertices in $\phi(N^+(v))$. 
To obtain an embedding, we must additionally guarantee that $\phi$ is injective.
Let $e_v$ be an arbitrary ordering of the tuple $N^+(v)$.
If the map constructed at the previous steps were well 
chosen so that the sum
\begin{align} \label{eq:important_sum}
	\sum_{v \in W_i} \omega_{\theta}(\phi(e_v); V_i)
\end{align}
is small, then we can proceed as in the proof of Theorem~\ref{thm:one_side_bounded}
to extend the map to $W_{i}$.

To briefly describe how to obtain control on \eqref{eq:important_sum},
we consider the case when all vertices in $v \in W_i$ satisfy 
$N^+(v) \subseteq W_{i+1}$. In this case, if the embedding $\phi|_{W_{i+1}}$ was chosen as a 
 uniform map from $W_{i+1}$ to $V_{i+1}$, then we would expect the sum above to be 
$|W_i|\mu_{1,\theta}(V_{i+1}^d; V_i)$. Hence, we can make \eqref{eq:important_sum}
small by finding sets $\{V_i\}_{i \in [k]}$ 
having small average defects in an appropriately defined way.
However, even if we are given such subsets $\{V_i\}_{i \in [k]}$, 
our assumption that
$\phi|_{W_{i+1}}$ is a uniform random map cannot be true.
Nevertheless, we show that \eqref{eq:important_sum} is not too far from
$\mu_{1,\theta}(V_{i+1}^d; V_i)$ by showing that our map is not too far from
the random uniform map. For example,
consider the simple case discussed above where all vertices in $v \in W_i$ satisfy 
$N^+(v) \subseteq W_{i+1}$. For a fixed vertex $v \in W_i$, let 
$w_1, \cdots, w_d$ be its neighbors in $N^+(v)$ and let
$e_v = (w_1, \cdots, w_d)$. If the embedding was chosen uniformly at random, then 
the image of $w_j$ would be chosen in the set $V_{i+1}$, but our algorithm
chooses its image in $N( \phi(e_{w_j}); V_{i+1})$ instead. 
Hence, compared to the random uniform map, the expected value of 
$\omega_{\theta}(\phi(e_v); V_{i+1})$ could be larger by at most a factor 
$\prod_{j \in [d]} \frac{|V_{i+1}|}{|N(\phi(e_{w_j}); V_{i+1})|}$.
Note that either $|N(\phi(e_{w_j}); V_{i+1})| \ge \theta$
or $\frac{|V_{i+1}|}{|N(\phi(e_{w_j}); V_{i+1})|} = \frac{|V_{i+1}|}{\theta} \omega_{\theta}(\phi(e_{w_j}); V_{i+1})$.
Therefore, by the inequality of arithmetic and geometric
means,
\[
	\prod_{j \in [d]} \frac{|V_{i+1}|}{|N(\phi(e_{w_j}); V_{i+1})|}
	\le \frac{1}{d}\sum_{j \in [d]} \left(\frac{|V_{i+1}|}{|N(\phi(e_{w_j}); V_{i+1})|}\right)^{d}
	\le \frac{1}{d}\sum_{j \in [d]} \frac{|V_{i+1}|^d}{\theta^d}\left(1 + \omega_{\theta}(\phi(e_{w_j}); V_{i+1})^d\right).
\]
The second term in the summand in expectation has value
$\mu_{d,\theta}(V_{i+2}^d; V_{i+1})$ and thus if this quantity is less than $1$ 
and everything worked as planned, 
then the right-hand side would be at most $2\left(\frac{|V_{i+1}|}{\theta}\right)^d$.
This means that in expectation, the sum \eqref{eq:important_sum} would be at most
$2\left(\frac{|V_{i+1}|}{\theta}\right)^d \cdot \mu_{1,\theta}(V_{i+1}^d; V_i)$.
In general, if $N^{+}(v) \not\subseteq W_{i+1}$, then there will be more dependencies between the random variables and it gets
trickier to control the events. However, the underlying idea, comparing our 
algorithm with the uniform random map and measuring its deviation using higher moments
of the average defect function, remains the same.
We will formalize this idea in Section~\ref{sec:embedding}.

\section{Embedding scheme} \label{sec:embedding}

In this section, we develop the embedding scheme.
The main new ingredient of our embedding algorithm compared to previous ones
lies in the randomness. We recall the proof of Theorem~\ref{thm:one_side_bounded} 
to illustrate the essence of this difference.
There, we embedded a graph $H$ with bipartition $W_1 \cup W_2$
into a graph $V$ with bipartition $V_1 \cup V_2$. 
We started by finding a set $A \subseteq V_2$ with small defect into $V_1$
using dependent random choice. This step is identical to that from previous applications
(except that we imposed a weaker bound on the defect function).
In the second step, we found a map $\phi$ from $W_2$ to $V_2$.
The main difference lies in this step. 
In previous applications, it was done greedily so that for all 
$v \in W_1$, the set $\phi(N(v))$ has many common neighbors.
For our purpose, the weaker bound from the first step makes this strategy not viable.
Therefore, we took a random map from $W_2$ to $V_2$ and used the first moment to find a map for which the number of common neighbors
is controlled by $\mu(A^d; V_1)$. 
Finally, we embedded $W_1$ in $V_1$ so that each vertex $v \in W_1$ gets mapped
to a vertex in the common neighbor of $\phi(N(v))$.
The difficulty in extending this algorithm to non-bipartite graphs lies in the fact that there is 
no clear distinction between the second and third steps. That is, when mapping a vertex
$v \in W_1$ to the common neighbor of $\phi(N(v))$, we must retain randomness
so that the process can be continued.

Let $G$ be a graph with disjoint vertex subsets $\{V_i\}_{i \in [k]}$
and $H$ be a graph with a vertex partition $\{W_i\}_{i \in [k]}$ into independent sets
where each vertex in $W_i$ has at most $d$ neighbors in $W_{i+1} \cup \cdots \cup W_k$
for all $i \in [k-1]$.
Let $\{\theta_i\}_{i \in [k-1]}$ be natural numbers.
For $i \in [k-1]$ and a vertex $x \in W_i$, define $N^+(x) = N(x) \cap \bigcup_{j \in [i+1,k]} W_j$.
Add edges to $H$ if necessary so that $|N^+(x)| = d$ for all vertices $x \in V(H) \setminus W_k$.
For $i \in [k]$ and $x \in W_i$, let $e_x$ be
an (arbitrary) ordered $d$-tuple of vertices formed from $N^+(x)$.
Define a random map $\psi : V(H) \rightarrow V(G)$ using the following `random greedy' process
\begin{itemize}
  \setlength{\itemsep}{1pt} \setlength{\parskip}{0pt}
  \setlength{\parsep}{0pt}
\item[1.] Take an injection from $W_{k}$ to $V_{k}$ uniformly at random.
\item[2.] For $i \in [k-1]$, given a map $\psi$ defined on $W_{i+1} \cup \cdots \cup W_{k}$, we then extend $\psi$ to $W_i$. Let $x_1, x_2, \cdots, x_{m}$ be the vertices in $W_{i}$ ordered so that $\omega_{\theta_i}(\psi(e_{x_j}); V_i)$ is decreasing in $j$ (ties are broken arbitrarily).
\item[3.] After embedding $x_1, \cdots, x_{j-1}$, define $e_j = e_{x_j}$ and $L_j = N(\psi(e_j); V_i) \setminus \{\psi(x_1), \cdots, \psi(x_{j-1})\}$ as the set of available vertices for $x_j$. Embed $x_j$ using the first valid option among the following:
\begin{itemize}  
  \setlength{\itemsep}{1pt} \setlength{\parskip}{0pt}
  \setlength{\parsep}{0pt}
\item[3-1.] If $N(\psi(e_j); V_i) = \emptyset$, then let $\psi(x_j)$ be a vertex in $V_i$ chosen uniformly at random. 
\item[3-2.] If $|L_j| < \frac{1}{2}|N(\psi(e_j);V_i)|$, then let $\psi(x_j)$ be a vertex in $N(\psi(e_j); V_i)$ 
chosen uniformly at random.
\item[3-3.] If $|L_j| \ge \frac{1}{2}|N(\psi(e_j);V_i)|$, then let $\psi(x_j)$ be a vertex in $L_j$ chosen uniformly at random. 
\end{itemize}
\end{itemize}
If we only run Steps 3-2 and 3-3 (but never Step 3-1)
throughout the embedding process then the resulting map is a homomorphism
from $H$ to $G$.
As we will later see, Step 3-1 can easily be ruled out if we have control on the defect function 
so one can safely assume that $\psi$ is always an homomorphism from $H$ to $G$.
Given that $\psi$ is a homomorphism, it suffices to show that $\psi$ is injective, in order to show that it is an embeding.
Thus, it will be crucial to understand when we will run Step 3-3 instead of Step 3-2.

\medskip

Throughout the proof we consider the defect function $\omega_{\theta}(Q; T)$
for various different choices of $\theta, Q$, and $T$.
Since these parameters
will always be a function of the vertex $x \in V(H)$ that we are about to embed, we make the following definitions to remove redundant information and simplify notation.

\begin{dfn} \label{def:parametrized_notation}
Suppose that we are given $\{V_i\}_{i \in [k]}$, $\{W_i\}_{i \in [k]}$, and $\{\theta_i\}_{i \in [k]}$.
For $i \in [k-1]$ and $x \in W_i$, make the following definitions:
\begin{itemize}
  \setlength{\itemsep}{1pt} \setlength{\parskip}{0pt}
  \setlength{\parsep}{0pt}
\item[(i)] $\theta_x := \theta_i$ and $V_x := V_i$. 
\item[(ii)] $\omega(x;\psi) := \omega_{\theta_x}(\psi(e_x); V_x)$.
\item[(iii)] $\mathcal{Q}_x := V_{i_1} \times \cdots \times V_{i_d}$
where $W_{i_1} \times \cdots \times W_{i_d}$ is the unique product space containing $e_x$. 
\item[(iv)] $\mu_s(x) := \mu_{s,\theta_x}(\mathcal{Q}_x; V_x)$ and  $\mu_{s} := \max_{x \in V(H)} \mu_s(x)$.
\item[(v)] $\gamma := \max\left\{1, \max_{i \in [k-1]} \frac{|V_i|}{\theta_i}\right\}$.
\end{itemize}
\end{dfn}

For $x \in V(H)$, if $\mu_s(x)$ is finite, then there are no $d$-tuples $Q \in \mathcal{Q}_x$ 
having $N(Q; V_x) = \emptyset$. Hence, we will never run Step 3-1 in such case.
For a vertex $x \in W_i$, the parameter $\theta_x = \theta_i$ is the defect
that is of interest when embedding $x$ in $V_i$, as whether we run Step 3-2 or 3-3
depends on $|N(\psi(e_x); V_x)|$ which is closely related to
$\omega(x;\psi) = \omega_{\theta_x}(\psi(e_x); V_x)$. Note that the enumeration of vertices
in Step 2 is determined by the values $\omega(x;\psi)$.
Further note that if the embedding algorithm chose the image of 
each vertex $y \in e_x$ as a random uniform vertex in $V_y$, then 
$\mu_s(x)$ would be the expected value of $\omega(x;\psi)^s$.
This is an important observation, since we will later see that 
even though $\psi$ does not
always choose a random uniform vertex,
$\gamma$ can be used to bound the difference between
the expected value of $\omega(x;\psi)^s$ defined by our random process
and by the random uniform case.
The following theorem establishes a sufficient condition for $\psi$ to be an embedding.

\begin{thm} \label{thm:success}
If $|V_k| \ge 2|W_k|$ and $\theta_i \ge 2|W_i|$ for all $i \in [k-1]$, 
then the probability that $\psi$ does not induce an embedding of $H$ in $G$ 
is at most $2^{2d+2}\gamma^{2d} \mu_{4d} \sum_{i \in [k-1]} \frac{|W_i|}{\theta_i}$.
\end{thm}

The rest of this section focuses on proving Theorem~\ref{thm:success}. The proof of Theorem~\ref{thm:main}
then follows by finding subsets of vertices of $G$ satisfying the condition of Theorem~\ref{thm:success}.

\subsection{Proof of Theorem~\ref{thm:success}}

For technical reasons, we consider auxiliary graphs.
Add a set $W_{k+1}$ of $d$ vertices to $H$, 
make the bipartite graph between $W_k$ and $W_{k+1}$ complete, and
between $W_i$ and $W_{k+1}$ empty for all $i \in [k-1]$. Denote the resulting graph
as $H'$.
Define $N^+(x) := N(x) \cap W_{k+1}$ for each $x \in W_k$. 
Note that now, for all $x \in V(H)$, we have $|N^{+}(x)| = d$.
Let $G'$ be a graph obtained from $G$ by adding
a set $V_{k+1}$ of $d$ vertices adjacent to all other vertices.
Define $\theta_k = |V_k|$ and note that $\omega_{\theta_k}(Q; V_k) = 0$ for all 
$Q \in V_{k+1}^d$.
Further note that for all $x \in W_k$, we have $\omega_{\theta_k}(x;\psi) = 0$ and $\mu_{s}(x) = 0$ 
regardless of the choice of 
$\psi$ and $s$ since $e_x \in W_{k+1}^d$. Note that Definition~\ref{def:parametrized_notation} now extends to $i=k$ as well. The parameters
for $i\le k-1$ do not change depending on whether we consider the graphs $H,G$,
or the graphs $H',G'$. Furthermore, $\gamma$ is identical when defined for the pair 
$G$ or in $G'$, since $\frac{|V_k|}{\theta_k} = 1$. 
Thus we abuse notation and use the parameters defined in 
Definition~\ref{def:parametrized_notation} as if they were defined for $H',G'$.

Throughout this section, we will apply the embedding scheme defined above to embed
$H'$ to $G'$, but with a slightly modified first step:
\begin{itemize}
  \setlength{\itemsep}{1pt} \setlength{\parskip}{0pt}
  \setlength{\parsep}{0pt}
\item[1'.] Take a map from $W_{k+1}$ to $V_{k+1}$ uniformly at random (instead of a random injection).
\end{itemize}
We then embed the rest of the graph using the same algorithm (where we also consider $i=k$ 
in Steps 2 and 3). Note that $\psi$ restricted
to the vertices $V(H)$ has the same distribution as the map previously defined without the
sets $W_{k+1}$ and $V_{k+1}$. 

The following lemma gives a sufficient condition for $\psi$ to be an embedding.

\begin{lem} \label{lem:random_greedy}
Let $G, H, G', H', \{V_i\}_{i \in [k+1]}, \{W_i\}_{i \in [k+1]}$, and $\psi$ be as described above.
Suppose that $\theta_i \ge 2|W_i|$ for all $i \in [k]$.
For a natural number $s$, let $\phi : V(H') \rightarrow V(G')$ be a map satisfying
$\BFP(\psi = \phi) > 0$ and
$\sum_{x \in W_i} \omega(x; \phi)^s \le \frac{1}{2}\theta_i$ for all $i \in [k]$.
Then $\phi |_{V(H)}$ is an embedding of $H$ in $G$.
\end{lem}
\begin{proof} 
As discussed in Section~\ref{sec:preliminaries}, we have
$\sum_{x \in W_i} \omega(x; \phi) \le \sum_{x \in W_i} \omega(x; \phi)^s$ and therefore it suffices
to consider the case $s=1$.
Let $\phi : V(H') \rightarrow V(G')$
be a map satisfying $\sum_{x \in W_i} \omega(x; \phi) \le \frac{1}{2}\theta_i$
for all $i \in [k]$.
Fix $i \in [k]$ and condition on the event that $\psi = \phi$ after mapping the vertices 
$W_{i+1}, \cdots, W_k$. 
Let $x_1, x_2, \cdots, x_m$ be the vertices in $W_i$ in decreasing order of
$\omega(x_j; \phi)$ and define $e_j = e_{x_j}$ for each $j \in [m]$. 
Note that this is the order that the vertices in $W_i$ will be
mapped to $V_i$. 
Consider the $j$-th step. Since $\sum_{x \in W_i} \omega(x; \phi)$ is finite, 
we have $N(\phi(e_j); V_i) \neq \emptyset$.
Hence, we will not run Step 3-1 when mapping $v_j$.

If $\omega(x_j; \phi) = 0$, then $|N(\phi(e_j); V_{i})| \ge \theta_i \ge 2|W_i|$
and thus we determine $\psi(x_j)$ according to Step 3-3. 
If $\omega(x_j; \phi) \neq 0$, then by how we ordered the vertices, 
we have $j \cdot \omega(x_j; \phi) \le \frac{1}{2}\theta_i$, and thus
\[
	\frac{\theta_i}{|N(\phi(e_j); V_i)|} = \omega(x_j; \phi) \le \frac{\theta_i}{2j},
\]
from which it follows that $|N(\phi(e_j) ; V_i)| \ge 2j$.
Since we embedded at most $j-1$ vertices of $W_i$ prior to $x_j$,
we determine $\psi(x_j)$ according to Step 3-3 when embedding $x_j$. 
Since $\BFP(\psi(x_j) = \phi(x_j)) > 0$ (conditioned on $\psi = \phi$ for all
previous vertices), we see that $\phi(x_j)$ is in $N(\phi(e_j); V_i)$ and
is distinct from all previously mapped vertices. 
Since the analysis applies to all steps of the embedding, we can conclude that
$\phi|_{V(H)}$ is injective, and thus an embedding.
\end{proof}

Lemma~\ref{lem:random_greedy} shows that it is crucial to control the quantity
 $\sum_{x \in W_i} \omega(x; \psi)^s$.
Note that the expected value of $\omega(x; \psi)^s$ is $\mu_{s}(x)$ if $\psi(e_x)$
were uniformly distributed in $\mathcal{Q}_x$. 
This was the case in the proof of Theorem~\ref{thm:one_side_bounded}, and there, thus 
we were able to conclude that $\sum_{x \in W_i} \omega(x; \psi)^s$ is small
quite straightforwardly.
Now our situation is more complicated, since $\psi(e_x)$ is no longer uniformly distributed in $\mathcal{Q}_x$.
We gain control on the sum by comparing our distribution with 
the uniform distribution.

Let $\nu$ be the distribution on the set of maps $\phi : V(H') \rightarrow V(G')$ 
obtained by the random greedy embedding algorithm defined above. 
For a set of vertices $I \subseteq V(H)$, we say that $\tilde{\nu}$ is 
a probability distribution obtained from $\nu$ by {\em neutralizing $I$}
if in the random greedy process above, every time we embed a vertex $x \in I$, 
instead of following Steps 3-1, 3-2, or 3-3, we choose the image of $x$
as a uniform random vertex in $V_x$. For example, if we neutralize all vertices, 
then the resulting distribution is a uniform distribution over all maps
$\phi : V(H') \rightarrow V(G')$ satisfying $\phi(W_i) \subseteq V_i$ for all $i \in [k+1]$.

From now on, we use $\psi$ to denote a random map from $V(H')$ to $V(G')$
whose distribution is determined by the probability measure under consideration.
In contrast, we use $\phi$ to denote fixed (non-random) maps. 
Recall that $\gamma = \max\left\{1, \max_{i \in [k]} \frac{|V_i|}{\theta_i} \right\}$.

\begin{lem} \label{lem:decouple}
Let $I_1 \subseteq V(H)$ be a subset of vertices and
let $\nu_1$ be the distribution obtained from $\nu$ by neutralizing $I_1$.
Let $X$ be a non-negative random variable depending only on the images of vertices in $J$ for some $J \subseteq V(H)$.
Define $I_2 = I_1 \cup J$ and define $\nu_2$
as the distribution obtained from $\nu$ by neutralizing $I_2$. Suppose that $t = |J \setminus I_1| \ge 1$. Then
\[
	\BBE_{\nu_1}[X]
	\le 2^{t}\gamma^{t} \BBE_{\nu_2}[X] +  2^{2t - 1} \gamma^{2t} \BBE_{\nu_2}[X^2] + \frac{1}{2t} \sum_{y \in J \setminus I_1}
		\BBE_{\nu_2}\left[ \omega(y; \psi)^{2t} \right].
\]
Therefore, if $X^2 \ge X$ holds with probability $1$, then 
\[
	\BBE_{\nu_1}[X]
	\le 2^{2t} \gamma^{2t} \BBE_{\nu_2}[X^2] + \frac{1}{2t} \sum_{y \in J \setminus I_1}
		\BBE_{\nu_2}\left[ \omega(y; \psi)^{2t} \right].
\]
\end{lem}

Note that the first term on the right-hand side, $\BBE_{\nu_2}[X^2]$, under the probability 
measure $\nu_2$, is determined completely
by a set of vertices whose images are chosen uniformly at random. Hence, this lemma allows us to
compare the distribution on maps defined by our random greedy algorithm with the random uniform map.

\begin{proof}
For each vertex $y \in I_2 \setminus I_1$, define $C_y(\phi) = \frac{2|V_y|}{|N(\phi(e_y); V_y)|}$. 
For a set $W \subseteq V(H')$ of size $|W|=t$, fix a map $\phi : W \rightarrow V(G')$. 
We use the notation $\psi_{t} = \phi$ to indicate the event that the random map $\psi$ obtained
after embedding the first $t$ vertices is $\phi$. Define $I_W = I_2 \setminus (I_1 \cup W)$. 
Let $\nu_W$ be the distribution obtained from $\nu_1$ by neutralizing $I_1 \cup I_W$. 
We prove that for all $W, t$ and $\phi$ as above,
\begin{align} \label{eq:rev_induction}
	\BBE_{\nu_1}\left[X\,\Big|\,\psi_{t} = \phi\right]
	\le \BBE_{\nu_W}\left[X \cdot  \prod_{y \in I_W} C_y(\psi) \,\Big|\,\psi_{t} = \phi\right].
\end{align}
We prove \eqref{eq:rev_induction} by (downward) induction on $t$. If $t = |V(H)|$, then $I_W = \emptyset$ 
and so the random variables on both sides of the inequality equal $X$. 
Therefore, \eqref{eq:rev_induction} trivially holds. 

Let us now investigate the value $t$ while assuming that \eqref{eq:rev_induction} is true for all larger values.
Let $W$ be a set of size $t$ and $\phi$ be a map defined on $W$. Conditioned on $\psi_t = \phi$, 
we know which vertex will be embedded next. Say that it is $x_{t+1} \in W_i$ for some $i \in [k]$. 
Define $W' = W \cup \{x_{t+1}\}$.
For each $z \in V_i$, let $\phi_z$ be the extension of $\phi$ obtained by defining $\phi(x_{t+1}) = z$.
Then
\begin{align*}
	\BBE_{\nu_1}\left[X\Big|\,\psi_{t} = \phi\right]	
	=&\,  \sum_{z \in V_i} \BFP_{\nu_1}\left(\psi_{t+1} = \phi_z \,|\, \psi_{t} = \phi \right)\BBE_{\nu_1}\left[X\,\Big|\,\psi_{t+1} = \phi_z\right].
\end{align*}
Therefore, by the inductive hypothesis, we have
\begin{align*}
	\BBE_{\nu_1}\left[X \Big|\,\psi_{t} = \phi\right]	
	\le&\,  \sum_{z \in V_i} \BFP_{\nu_1}\left(\psi_{t+1} = \phi_z \,|\, \psi_{t} = \phi \right)
		\BBE_{\nu_{W'}}\left[X \cdot  \prod_{y \in I_{W'}} C_y(\psi) \,\Big|\,\psi_{t+1} = \phi_z\right] \\
	=&\,  \sum_{z \in V_i} \BFP_{\nu_1}\left(\psi_{t+1} = \phi_z \,|\, \psi_{t} = \phi \right)
		\BBE_{\nu_{W}}\left[X \cdot  \prod_{y \in I_{W'}} C_y(\psi) \,\Big|\,\psi_{t+1} = \phi_z\right],
\end{align*}
where the second equality follows since the distribution of $\nu_{W'}$ and $\nu_{W}$ differ only
on the image of $x_{t+1}$ which is fixed once we condition on $\psi_{t+1} = \phi_z$.

If $I_{W'} = I_{W}$, then $x_{t+1} \notin I_2$, and thus $\BFP_{\nu_1}(\psi_{t+1} = \phi_z | \psi_{t} = \phi) = \BFP_{\nu_W}(\psi_{t+1} = \phi_z \,|\,\psi_{t}= \phi)$. Therefore, the right-hand side above is
\begin{align*}
	\sum_{z \in V_i} \BFP_{\nu_{W}}\left(\psi_{t+1} = \phi_z \,|\, \psi_t = \phi \right)
		\BBE_{\nu_{W}}\left[X \cdot  \prod_{y \in I_{W}} C_y(\psi) \,\Big|\,\psi_{t+1} = \phi_z\right]
	=&\,\BBE_{\nu_{W}}\left[X \cdot  \prod_{y \in I_{W}} C_y(\psi) \,\Big|\,\psi_{t} = \phi\right],
\end{align*}
proving our claim.
On the other hand if $I_{W'} \neq I_{W}$, then $I_{W'} = I_W \cup \{x_{t+1}\}$
and $x_{t+1} \in J \setminus I_1$. Note that the image of $x_{t+1}$, under the probability distribution $\nu_1$, will be chosen uniformly in a set of size
$|V_i|$ (Case 3-1), $|N(\phi(e_{x_{t+1}}); V_i)|$ (Case 3-2), or at least $\frac{|N(\phi(e_{x_{t+1}}); V_i)|}{2}$ (Case 3-3). 
Since $\BFP_{\nu_{W}}(\psi_{t+1} = \phi_z \,|\,\psi_{t}= \phi) = \frac{1}{|V_{i}|}$,
it follows that
\[
	\BFP_{\nu_1}(\psi_{t+1} = \phi_z \,| \,\psi_t = \phi)
	\le \frac{2}{|N(\phi(e_{x_{t+1}}); V_i)|}
	= C_{x_{t+1}}(\phi) \cdot \BFP_{\nu_{W}}(\psi_{t+1} = \phi_z \,|\,\psi_{t}= \phi).
\]
Therefore,
\begin{align*}
	\BBE_{\nu_1}\left[X \Big|\,\psi_{t} = \phi\right]
	\le&\, \sum_{z \in V_i} \BFP_{\nu_1}\left(\psi_{t+1}= \phi_z \,|\, \psi_{t} = \phi \right)
			\BBE_{\nu_{W}}\left[X \cdot  \prod_{y \in I_{W'}} C_y(\psi) \,\Big|\,\psi_{t+1} = \phi_z\right] \\
	\le &\, \sum_{z \in V_{i}} C_{x_{t+1}}(\phi) \cdot \BFP_{\nu_{W}}\left(\psi_{t+1} = \phi_z \,|\, \psi_{t} = \phi \right) \BBE_{\nu_{W}}\left[X \cdot  \prod_{y \in I_{W'}} C_y(\psi) \,\Big|\,\psi_{t+1} = \phi_z\right].
\end{align*}
Since $C_{x_{t+1}}(\psi) = \frac{2|V_{x_{t+1}}|}{|N(\psi(e_{x_{t+1}}); V_{x_{t+1}})|}$ and 
$e_{x_{t+1}} \in \left(\bigcup_{j \ge i+1} W_{j}\right)^{d}$, the value of $C_{x_{t+1}}(\psi)$
is determined once we condition on $\psi_{t+1} = \phi_z$.
Since $I_{W} = I_{W'} \cup \{v_{t+1}\}$, we have
\begin{align*}
	\BBE_{\nu_1}\left[X\Big|\,\psi_{t} = \phi\right]
	\le &\, \sum_{z \in V_{i}} \BFP_{\nu_{W}}\left(\psi_{t+1} = \phi_z \,|\, \psi_{t} = \phi \right) \BBE_{\nu_{W}}\left[X \cdot  \prod_{y \in I_{W}} C_y(\psi) \,\Big|\,\psi_{t+1} = \phi_z\right] \\
	= &\, \BBE_{\nu_{W}}\left[ X \cdot  \prod_{y \in I_{W}} C_y(\psi) \,\Big|\, \psi_{t} = \phi \right].
\end{align*}
This proves the claim. 

The inequality \eqref{eq:rev_induction} for $t = 0$ gives
$\BBE_{\nu_1}[X] \le \BBE_{\nu_2}\left[X \cdot  \prod_{y \in I_2 \setminus I_1} C_y(\psi) \right]$.
Note that $C_y(\psi) = \frac{2|V_y|}{|N(\psi(e_y); V_y)|} \le 2|V_y| \cdot \max\left\{\frac{1}{\theta_y}, \frac{\omega(y;\psi)}{\theta_y}\right\}$.
Therefore, if $|I_2 \setminus I_1| = t$, then
\begin{align*}
	\prod_{y \in I_2 \setminus I_1} C_y(\psi)
	\le&\, 2^t \prod_{y \in I_2 \setminus I_1} |V_y| \left( \frac{\max\{1, \omega(y;\psi)\}}{\theta_y} \right) \\
	\le&\, \left(2^t \prod_{y \in I_2 \setminus I_1} \frac{|V_y|}{\theta_y} \right) \frac{1}{t}\sum_{y \in I_2 \setminus I_1} \left( \max\{1, \omega(y;\psi) \}\right)^t
	\le 2^t \gamma^t \frac{1}{t}\sum_{y \in I_2 \setminus I_1} (1 + \omega(y;\psi)^t).
\end{align*}
Hence, for $C = 2^t \gamma^t$,
\begin{align*}
	\BBE_{\nu_1}[X]
	\le&\, \BBE_{\nu_2}\left[CX \right]
	+ \frac{1}{t} \sum_{y \in I_2 \setminus I_1} \BBE_{\nu_2}\left[CX \cdot \omega(y;\psi)^{t}\right] \\
	\le&\, \BBE_{\nu_2}\left[CX \right]
	+ \frac{1}{t} \sum_{y \in I_2 \setminus I_1} 
	\frac{1}{2} \left(\BBE_{\nu_2}\left[C^2 X^2\ + \omega(y;\psi)^{2t}\right] \right) \\
	\le&\, \BBE_{\nu_2}\left[CX \right] + \frac{1}{2} \BBE_{\nu_2}\left[C^2X^2 \right]
	+ \frac{1}{2t} \sum_{y \in I_2 \setminus I_1} 
	\BBE_{\nu_2}\left[\omega(y;\psi)^{2t}\right]. \qedhere
\end{align*}
\end{proof}

We plan to gain control on the defects $\omega(x; \psi)$ by using Lemma~\ref{lem:decouple}. 
As explained above, the first term in the right-hand side of Lemma~\ref{lem:decouple} gives a 
direct comparison between our process and the random uniform map.
However, the second term in the right-hand side of Lemma~\ref{lem:decouple} is still problematic
since it is in general determined by vertices that are not yet neutralized.
We repeatedly apply Lemma~\ref{lem:decouple} to further gain control on these terms. 
As we proceed, the set of neutralized vertices further propagates and eventually
there will be no vertices left to be neutralized.

\begin{lem} \label{lem:defect_bound}
For all $x \in V(H)$, 
$\BBE\left[ \omega(x; \psi)^{2d} \right]\le 2^{2d+1} \gamma^{2d} \mu_{4d}$.
\end{lem}
\begin{proof}
For a vertex $x \in W_i$, let $T_0$ be the vertex-weighted rooted tree 
labelled by vertices in $V(H)$ with a 
single root vertex labelled $x$ having weight $1$. 
For $j \ge 0$, suppose that we constructed a weighted rooted labelled tree $T_{j}$, 
and let $L_{j} = V(T_j) \setminus V(T_{j-1})$.
For a node $a \in L_{j}$, define $y_a \in V(H)$ as the label of $a$,
and $P(a)$ as the set of vertices 
on the path from the root to the parent node of $a$ in the tree $T_{j}$. 
Let $I(a) = W_{k+1} \cup \bigcup_{b \in P(a)} N^+(y_b)$ and let
$F(a) = N^+(y_a) \setminus I(a)$. For each vertex $y \in F(a)$, add a child
to $a$ labelled $y$ and let its weight be $\frac{\sigma(a)}{2|F(a)|}$, where
$\sigma(a)$ is the weight of $a$ in $T_{j}$. Let $T_{j+1}$ be the tree obtained 
by doing this process for all $a \in L_j$. 
Note that the process eventually stops since each path from a root to a leaf 
has labels of the form $(y_{i_1}, y_{i_2}, \cdots, y_{i_s})$ for vertices $y_{i_j} \in W_{i_j}$
satisfying $i_1 < i_2 < \cdots < i_s$. Let $T_t$ be the final tree, i.e., $t$ is the minimum
integer for which $T_t = T_{t+1}$.
We may assume that $T_1 \subseteq T_2 \subseteq \cdots \subseteq T_t$. 

Define $\gamma' = 2^{2d} \gamma^{2d}$.
Let $\nu$ be the probability measure on maps $\psi : V(H') \rightarrow V(G')$ induced by
our random embedding algorithm.
For a tree $T_j$ and its node $a$, let $\nu_a$ denote the probability measure 
obtained from $\nu$ by neutralizing $I(a)$.
We claim that for each $j \ge 0$,
\[
	\BBE_\nu\left[ \omega(x; \psi)^{2d} \right]
	\le \sum_{a \in V(T_{j-1})} \sigma(a) \gamma' \mu_{4d} 
	+ \sum_{a \in L_j} \sigma(a) \BBE_{\nu_a}[\omega(y_a;\psi)^{2d}],
\]
where for $j=0$ we let $T_{-1}$ be the empty graph.
The root node $r$ is the unique node in $T_0$ and has weight $\sigma(r) = 1$.
Moreover, $\nu_r$ and $\nu$ have identical distribution since $I(r) = W_{k+1}$
and we defined the map from $W_{k+1}$ to $V_{k+1}$ to be a random uniform map. 
Therefore, the inequality above holds for $j=0$ (in fact equality holds).

Suppose that the claim holds for some $j \ge 0$.
For a vertex $a \in L_j$, let $\Gamma(a)$ be the set of children of $a$ in $T_{j+1}$ 
and let $\nu_0$ be the probability measure obtained from $\nu_a$ by neutralizing
the vertices in $N^+(a)$. 
If $\Gamma(a) = \emptyset$, then $N^+(y_a) \subseteq I(a)$ and thus
$\BBE_{\nu_a}[\omega(y_a;\psi)^{2d}] \le \mu_{2d} \le \mu_{4d}$.
Otherwise if $\Gamma(a) \neq \emptyset$, then for all $b \in \Gamma(a)$, we have $\nu_0 = \nu_b$. 
Therefore, by Lemma~\ref{lem:decouple}, we have
\begin{align} \label{eq:propagate}
	\BBE_{\nu_a}[\omega(y_a;\psi)^{2d}]
	\le&\, \gamma' \cdot \BBE_{\nu_0}[\omega(y_a;\psi)^{4d}]
		+ \frac{1}{2|\Gamma(a)|} \sum_{b \in \Gamma(a)} \BBE_{\nu_{b}}[\omega(y_b;\psi)^{2|\Gamma(a)|}].
\end{align}
Since the images of $N^+(a)$ are neutralized in the measure $\nu_{0}$, by definition we have $\BBE_{\nu_0}[\omega(y_a;\psi)^{4d}] \le \mu_{4d}$. 
Since $\gamma' \ge 1$, 
we will use the same bound \eqref{eq:propagate} for nodes $a \in L_j$ having $\Gamma(a) = \emptyset$
with the understanding that the second term equals zero for such nodes.
Therefore, by the inductive hypothesis and the fact $|\Gamma(a)| = |F(a)| \le d$ for all $a \in L_j$, we have
\begin{align*}
	\BBE_\nu\left[ \omega(x; \psi)^{2d} \right]
	\le&\, \sum_{a \in V(T_{j-1})} \sigma(a) \gamma' \mu_{4d}
	+ \sum_{a \in L_{j}} \sigma(a) \BBE_{\nu_a}[\omega(y_a;\psi)^{2d}] \\
	\le&\, \sum_{a \in V(T_{j-1})} \sigma(a) \gamma' \mu_{4d}
	+ \sum_{a \in L_{j}} \sigma(a) \left(\gamma' \cdot \mu_{4d} + \frac{1}{2|F(a)|}\sum_{b \in \Gamma(a)} \BBE_{\nu_b}[\omega(y_b;\psi)^{2d}]\right) \\
	=&\, \sum_{a \in V(T_{j})} \sigma(a) \gamma' \mu_{4d}
	+ \sum_{a \in L_{j+1}} \sigma(a) \BBE_{\nu_a}[\omega(y_a;\psi)^{2d}],
\end{align*}
thus proving the claim.

Since $T_t = T_{t+1}$, in the end we see that 
$\BBE_\nu\left[ \omega(x; \psi)^{2d} \right] \le \sum_{a \in V(T_{t})} \sigma(a) \gamma' \mu_{4d}$.
For each $j \ge 0$, let $U_{j} \subseteq V(T_t)$ be the set of nodes of $T_t$
that are at distance $j$ from the root node.
Since $\sigma(a) = \sum_{b \in \Gamma(a)} 2\sigma(b)$ holds for every
non-leaf node $a \in V(T_t)$, a simple recursive argument implies 
$\sum_{a \in U_{j}} 2^{j} \sigma(a) \le 1$, or equivalently
$\sum_{a \in U_{j}} \sigma(a) \le 2^{-j}$ for all $j \ge 0$. Since $T_t$ is a 
finite tree, this implies
\[
	\sum_{a \in V(T_t)} \sigma(a)
	\le \sum_{j \ge 0} \sum_{a \in U_j} \sigma(A)
	\le \sum_{j \ge 0} 2^{-j} = 2.
\]
Therefore,
\begin{align*}
	\BBE_\nu\left[ \omega(x; \psi)^{2d} \right]
	\le&\, \sum_{a \in V(T_{t})} \sigma(a) \gamma' \mu_{4d} 
	\le 2^{2d+1}\gamma^{2d}\mu_{4d}. \qedhere
\end{align*}
\end{proof}

Lemma~\ref{lem:defect_bound} can be used to control sums of the form 
$\sum_{x \in W_i} \omega(x; \psi)^s$ that appear in Lemma~\ref{lem:random_greedy}.
Using it, we can prove Theorem~\ref{thm:success} 
which gives a quantifiable condition which
will guarantee the existence of an embedding. 
It is equivalent to the following form since we defined
$\theta_k = |V_k|$ when adding the artificial set $W_{k+1}$ at the beginning 
of this subsection. 

\vspace{0.2cm}
\noindent \textbf{Theorem \ref{thm:success}}. 
If $\theta_i \ge 2|W_i|$ for all $i \in [k]$, then 
the probability that $\psi$ does not induce an embedding of $H$ in $G$ 
is at most $2^{2d+2}\gamma^{2d} \mu_{4d} \sum_{i \in [k-1]} \frac{|W_i|}{\theta_i}$.
\vspace{-0.2cm}

\begin{proof}
Define $\lambda_i = \sum_{x \in W_i} \omega(x;\psi)^{2d}$ for each $i \in [k]$.
If $\sum_{i \in [k]} \frac{\lambda_i}{\theta_i} < \frac{1}{2}$, then
we have $\lambda_i < \frac{1}{2}\theta_i$ for all $i \in [k]$, and therefore by 
Lemma~\ref{lem:random_greedy}, $\psi$ induces an embedding of $H$ in $G$.
Note that $\lambda_k = 0$ since $\omega(x;\psi) = 0$ for all $x \in W_k$.
Thus, $\psi$ may not induce an embedding of $H$ in $G$
only if $\sum_{i \in [k]} \frac{\lambda_i}{\theta_i} = \sum_{i \in [k-1]} \frac{\lambda_i}{\theta_i} \ge \frac{1}{2}$.
By Markov's inequality the probability of this event is at most 
$2\BBE\left[ \sum_{i \in [k-1]} \frac{\lambda_i}{\theta_i} \right]$. 
By Lemma~\ref{lem:defect_bound}, for all $x \in V(H)$, we have
$\BBE\left[ \omega(x;\psi)^{2d} \right] \le 2^{2d+1} \gamma^{2d} \mu_{4d}$.
Therefore,
$\BBE\left[ \lambda_i \right] \le |W_i| \cdot 2^{2d+1} \gamma^{2d} \mu_{4d}$, and
\[
	2\BBE\left[ \sum_{i \in [k-1]} \frac{\lambda_i}{\theta_i} \right]
	\le 2 \cdot \sum_{i \in [k-1]} 2^{2d+1} \frac{|W_i|}{\theta_i} \gamma^{2d} \mu_{4d}
	\le 2^{2d+2}\gamma^{2d} \mu_{4d} \sum_{i \in [k-1]} \frac{|W_i|}{\theta_i}.	\qedhere
\]
\end{proof}

\section{Proof of the main theorem} \label{sec:proof_main}

In this section, we prove Theorems~\ref{thm:main} and \ref{thm:bipartite}.
We first prove the bipartite case, Theorem~\ref{thm:bipartite}, for which a density-type embedding result holds
and the proof is slightly more simple.
We then prove our main theorem, Theorem~\ref{thm:main}, in Subsection~\ref{subsec:degenerate}.

The following lemma is a slight variant of Lemma~\ref{lem:drc_simple} for non-bipartite graphs. 
Since we need to add further conditions to the conclusion, 
instead of stating the outcome of dependent random choice in terms of a particular set,
we state it in terms of the expected value of the random variables of interest. 
We omit its proof since it follows from the proof of 
Lemma~\ref{lem:drc_simple} after making straightforward modifications.

\begin{lem} \label{lem:drc}
Let $d, s$ and $t$ be natural numbers satisfying $t \ge s$, and
$\eta, \alpha$ be positive real numbers.
Let $G$ be a graph with two sets $V_1, V_2 \subseteq V(G)$ satisfying $e(V_1, V_2) \ge \alpha|V_1||V_2|$.
Let ${\bf X}$ be a $t$-tuple in $V_1^t$ chosen uniformly at random and let ${\bf A} = N({\bf X}; V_2)$. 
Then for all $\theta \le \eta \alpha^{d}|V_1|$,
\[
	\BBE[{\bf A}^d] \ge \alpha^{dt}|V_2|^d
	\quad \textrm{and} \quad
	\BBE\left[ \sum_{Q \in {\bf A}^d} \omega_{\theta}(Q; V_1)^s \right] \le |V_2|^d \cdot \eta^t \alpha^{dt}.
\]
In particular, this implies the existence of a set $A$ such that 
$|A| \ge \frac{1}{2}\alpha^{t}|V_2|$ and $\mu_{s}(A_2^{d}; V_1) \le 2\eta^{t}$.
\end{lem}

The outline of the proof in both cases, bipartite and non-bipartite, are the same.
We first prepare the host graph by repeatedly applying Lemma~\ref{lem:drc} to 
find a collection of $r$ sets $\{A_j\}_{j \in [r]}$
that have small average defect towards each other. 
This framework was first developed by Kostochka and Sudakov in \cite{KoSu} and its 
variants have been used in much subsequent work \cite{FoSu09, FoSu09-2, Lee}. 
Although the statement we need easily
follows from the proof of Fox and Sudakov's application \cite{FoSu09-2}, we
cannot use their work as a blackbox since we must 
control higher moments of the defect function.
The following proposition summarizes the main observation 
underlying this framework.

\begin{prop} \label{prop:defect_transfer}
Let $G$ be a graph with sets $V_1, A_2 \subseteq V(G)$. Let ${\bf X}$ be a $t$-tuple in $A_2^t$ chosen uniformly at random and let 
${\bf A}_1 = N({\bf X}; V_1)$. Then
\[
	\BBE\left[\mu_{s,\theta}(A_2^d; {\bf A}_1)\right]	= \mu_{s,\theta}(A_2^{d+t}; V_1).
\]
\end{prop}
\begin{proof}
For all $Q \in A_2^d$, we have $N(Q;{\bf A}_1) = N(Q \cup {\bf X}; V_1)$. 
Therefore, $\omega_{\theta}(Q; {\bf A}_1)^s = \omega_{\theta}(Q \cup {\bf X}; V_1)^s$, where
we consider $Q \cup {\bf X}$ as a $(d+t)$-tuple. 
Hence, 
\begin{align*}
	\BBE\left[\omega_{\theta}(Q; {\bf A}_1)^s\right]
	= \BBE\left[\omega_{\theta}(Q \cup {\bf X}; V_1)^s\right]
	=&\, \sum_{Y \in A_2^t} \omega_{\theta}(Q \cup Y; V_1)^s \cdot \BFP({\bf X} = Y).
\end{align*}
Since $\BFP({\bf X} = Y) = \frac{1}{|A_2|^t}$, it follows that
\begin{align*}
	\BBE\left[\mu_{s,\theta}(A_2^d; {\bf A}_1) \right]
	= \BBE\left[\frac{1}{|A_2|^d}\sum_{Q \in A_2^d} \omega_{\theta}(Q; {\bf A}_1)^s \right]
	= \frac{1}{|A_2|^{d+t}}\sum_{Q \in A_2^d} \sum_{Y \in A_2^t} \omega_{\theta}(Q \cup Y; V_1)^s.
\end{align*}
The last expression equals $\mu_{s,\theta}(A_2^{d+t}; V_1)$ by definition.
\end{proof}

By repeatedly applying Proposition~\ref{prop:defect_transfer} following
techniques developed in previous works, we can obtain
a collection of $r$ sets $\{A_j\}_{j \in [r]}$
that have small average defect towards each other.
Once we are given such a collection,
we use the following lemma 
producing sets $\{V_{i}^{(j)}\}$ that are needed for the embedding scheme of Section~\ref{sec:embedding}.
Its proof follows from standard probabilistic methods but is rather technical and 
thus will be given separately in another section.
For an $r$-tuple of sets $\{A_j\}_{j \in [r]}$, we define
$A_{-j} = \bigcup_{j' \in [r] \setminus \{j\}} A_{j'}$ for each $j \in [r]$. 

\begin{lem} \label{lem:breakdownG}
Let $k, d, s, r$ be fixed natural numbers satisfying $r \ge 2$, $s \ge 4d$,
and $\varepsilon, \varepsilon'$ be fixed positive real numbers.
Let $m$ be a natural number satisfying $m \ge 2^{200}(\varepsilon')^{-80}\varepsilon^{-80d}r^{160d}d^{40}k^{40}s^{8}$, and $\theta$ be a natural number satisfying $\theta \ge \varepsilon m$.
Let $p_i$ for $i \in [k]$ be positive real numbers  
satisfying $\sum_{i \in [k]} p_i \le 1$ and $p_i \ge m^{-1/(10d)}$ for all $i \in [k]$.
Suppose that $\{A_j\}_{j \in [r]}$ are vertex subsets 
with $\varepsilon m \le |A_j| \le m$, 
satisfying $\mu_{s, \theta}(A_{-j}^{d}; A_j) < \frac{1}{2}$ for all $j \in [r]$.
Define $\theta_i = \frac{1}{2r}p_i \theta$ for all $i \in [k]$.
Then there exist sets $\{V_{i}^{(j)}\}_{(i,j) \in [k] \times [r]}$ 
satisfying the following conditions:
\begin{itemize}
  \setlength{\itemsep}{1pt} \setlength{\parskip}{0pt}
  \setlength{\parsep}{0pt}
\item[(i)] For all $i \in [k]$ and $j \in [r]$, we have $V^{(j)}_{i} \subseteq A_j$ and $|V^{(j)}_{i}| \le p_i |A_j|$.
\item[(ii)] For all $((i, j), (i_1, j_1), \cdots, (i_d, j_d)) \in ([k] \times [r])^{d+1}$ satisfying $j_1, \cdots, j_d \neq j$,
\[
	\mu_{s,\theta_{i}}\left(\prod_{a \in [d]} V^{(j_a)}_{i_a}; V^{(j)}_{i}\right) \le \max\left\{\varepsilon', 8r^d\varepsilon^{-d} \mu_{s,\theta}(A_{-j}^d; A_j)\right\}.
\]
Moreover, if $r=2$, then the factor $r^d \varepsilon^{-d}$ can be replaced by $1$.
\item[(iii)] For all $(i,j) \neq (i',j')$ the sets $V_i^{(j)}$ and $V_{i'}^{(j')}$ are disjoint.
\end{itemize}
\end{lem}

We would like to note that the improvement for $r=2$ in item $(ii)$ plays an important role in the proof of Theorem~\ref{thm:bipartite_density_embed}, as it allows us to prove a 
single exponential bound for bipartite graphs (as opposed to the double exponential bound for general graphs).

\subsection{Bipartite graphs} \label{subsec:bip_degenerate}

We first find a pair of sets $(A_1, A_2)$ for which the $d$-tuples in $A_1^d$ have small
average defect into $A_2$, and vice versa for $d$-tuples in $A_2$. 
This will be achieved by applying Lemma~\ref{lem:drc} twice. First application will give a
set $A_2$ such that the $d$-tuples in $A_2^d$ have small average defect into $V(G)$. 
We then apply the lemma again by choosing a random set ${\bf X} \subseteq A_2$, 
to obtain a set $A_1$ whose $d$-tuples have small average defect into $A_2$. 
Proposition~\ref{prop:defect_transfer} can be used to show that
this pair of sets has the claimed properties.

\begin{lem} \label{lem:drc_both_prep}
Let $m, d, s, t$ be natural numbers and $\alpha, \eta$ be positive real numbers
satisfying $t \ge s$ and $\eta \le \frac{1}{16}\alpha^{2d}$. 
Let $G$ be a bipartite graph of minimum degree at least $\alpha m$ with vertex 
partition $V_1 \cup V_2$ where $|V_1| = |V_2| = m$.
Then there exist sets $A_1 \subseteq V_1$ and $A_2 \subseteq V_2$ satisfying the following properties:
\begin{itemize}
  \setlength{\itemsep}{1pt} \setlength{\parskip}{0pt}
  \setlength{\parsep}{0pt}
\item[(i)] $|A_i| \ge \frac{1}{4}\alpha^{t}m$ for both $i=1,2$, 
\item[(ii)] for all $\theta \le \frac{1}{2}\eta \alpha^{d+t}m$, $\mu_{s,\theta}(A_1^d; A_2) \le 2\eta^{t/2}$ and $\mu_{s,\theta}(A_2^d; A_1) \le 2\eta^{t/2}$.
\end{itemize}
\end{lem}
\begin{proof}
Since $\omega_{\theta'} \le \omega_{\theta}$ holds for all $\theta' \le \theta$, it suffices to consider
the case when $\theta = \frac{1}{2}\eta \alpha^{d+t}m$.
Throughout the proof, we will consider $\omega$ and $\mu$ with this fixed value of $\theta$
and hence will omit $\theta$ from the subscripts.
Since $\theta \le \eta \alpha^{d+t}|V_1|$, we can apply Lemma~\ref{lem:drc} with
$d_{\ref{lem:drc}} = d+t$ to find a set
$A_2 \subseteq V_2$ of size $|A_2| \ge \frac{1}{2}\alpha^{t}|V_2|$
such that $\mu_{s}(A_2^{d+t}; V_1) \le 2\eta^{t}$.

By the minimum degree condition on $G$, we know that the subgraph of $G$ induced on 
$V_1 \cup A_2$ has density at least $\alpha$. 
Let ${\bf X}$ be a $t$-tuple in $A_2^{t}$ chosen uniformly at random and let ${\bf A}_1 = N({\bf X})$. 
Since $\theta \le \eta \alpha^{d}|A_2|$, Lemma~\ref{lem:drc} with 
$(V_1)_{\ref{lem:drc}} = A_2$ and $(V_2)_{\ref{lem:drc}} = V_1$ implies
\[
	\BBE\left[|{\bf A}_1|^d - \frac{1}{2\eta^t} \sum_{Q \in {\bf A}_1^{d}} \omega(Q; A_2)^s\right]
	\ge \frac{1}{2} \alpha^{dt} |V_1|^d.
\]
By Proposition~\ref{prop:defect_transfer}, it follows that
$\BBE[\mu_{s}(A_2^d; {\bf A}_1)]= \mu_{s}(A_2^{d+t}; V_1) \le 2\eta^{t}$.
Since $\eta^{t/2} \le \frac{1}{4}\alpha^{dt}$,
\[
	\BBE\left[|{\bf A}_1|^d - \frac{1}{2\eta^t} \sum_{Q \in {\bf A}_1^{d}} \omega(Q; A_2)^s
		- \frac{|V_1|^d}{2\eta^{t/2}} \mu_{s}(A_2^d; {\bf A}_1)\right]
	\ge \frac{1}{4}\alpha^{dt} |V_1|^d.
\]
Let $X$ be a particular choice of ${\bf X}$ for which the random variable on the left-hand side
becomes at least as large as its expected value, and let $A_1$ be ${\bf A}_1$ for this choice of ${\bf X}$.
First, $|A_1|^d \ge \frac{1}{4}\alpha^{dt}|V_1|^d$ implies $|A_1| \ge \frac{1}{4}\alpha^{t}|V_1|$.
Second,
$|A_1|^d - \frac{1}{2\eta^t} \sum_{Q \in A_1^{d}} \omega(Q; A_2)^s \ge 0$ implies
$\mu_{s}(A_1^d; A_2) \le 2\eta^t$.
Third,
$|A_1|^d - \frac{|V_1|^d}{2\eta^{t/2}} \mu_{s}(A_2^d; A_1) \ge 0$ implies
$\mu_{s}(A_2^d; A_1) \le 2\eta^{t/2}$.
\end{proof}

The proof of the bipartite case of the Burr-Erd\H{o}s conjecture follows by combining
Lemma~\ref{lem:drc_both_prep} with Lemma~\ref{lem:breakdownG}, and then using the embedding
scheme from Section~\ref{sec:embedding}.
Even though the theorem below is stated for graphs of large minimum degree, the density-embedding
theorem result follows since every graph on $n$ vertices of density at least $\alpha$
contains a subgraph of minimum degree at least $\frac{1}{2}\alpha n$.

\begin{thm} \label{thm:bipartite_density_embed}
There exists a positive real number $c$ such that
for every natural number $d$ and positive real number $\alpha$, 
the following holds for all $m$ satisfying $m \ge \alpha^{-cd^2}$.
If $G$ is a graph on $m$ vertices with minimum degree at least $\alpha m$, then it is 
universal for the family of 
$d$-degenerate bipartite graphs on at most $d^{-1}2^{-21}\alpha^{72d}m$ vertices.
\end{thm}
\begin{proof}
Define $n = d^{-1}2^{-21}\alpha^{72d}m$.
Let $H$ be a $d$-degenerate bipartite graph on at most $n$ vertices. 
By Lemma~\ref{lem:split}, there exists a vertex partition $V(H) = \bigcup_{i \in [k]} (W_i^{(1)} \cup W_i^{(2)})$
satisfying the following properties:
\begin{itemize}
  \setlength{\itemsep}{1pt} \setlength{\parskip}{0pt}
  \setlength{\parsep}{0pt}
\item[(i)] $k \le \log_{2} n$,
\item[(ii)] for all $(i,j) \in [k] \times [2]$, we have $|W_{i}^{(j)}| \le 2^{-i+1}n$,
\item[(iii)] both $\bigcup_{i \in [k]} W_{i}^{(1)}$ and $\bigcup_{i \in [k]} W_{i}^{(2)}$ are independent sets, and 
\item[(iv)] for all $(i,j) \in [k] \times [2]$, each vertex $v \in W_i^{(j)}$ has at most $4d$ neighbors in
$\bigcup_{i'\ge i} W_{i'}^{(1)} \cup W_{i'}^{(2)}$.
\end{itemize}

Define $t = s = 32d$ and $\eta = \frac{1}{128}\alpha^{36d}$. 
Define $\theta = \frac{1}{8}\eta \alpha^{4d+t}m$.
We can use the concentration of hypergeometric distribution
(see, for example, \cite{JaLuRu})
to find a partition $V_1 \cup V_2$ of $V(G)$ 
for which $|V_1| = |V_2| = \frac{m}{2}$ and the bipartite subgraph  
induced on $V_1 \cup V_2$ has minimum degree at least $2^{-1/(36d)}\alpha m$.
Since $\eta \le \frac{1}{16} (2^{-1/(36d)}\alpha)^{8d}$ and
$\theta = \frac{1}{2}\eta \cdot (2^{-1/(36d)}\alpha)^{4d+t} \cdot \frac{m}{2}$,
we can apply Lemma~\ref{lem:drc_both_prep} with $d_{\ref{lem:drc_both_prep}} = 4d$ 
and $\alpha_{\ref{lem:drc_both_prep}} = 2^{-1/(36d)}\alpha$ to
obtain sets $A \subseteq V_1$ and $B \subseteq V_2$ satisfying the following properties:
\begin{itemize}
  \setlength{\itemsep}{1pt} \setlength{\parskip}{0pt}
  \setlength{\parsep}{0pt}
\item[(i)] $|A|, |B| \ge \frac{1}{4} 2^{-t/(36d)}\alpha^{t} \frac{m}{2} \ge \frac{1}{16}\alpha^t m$, and 
\item[(ii)] $\mu_{s,\theta}(A^{4d}; B) \le 2\eta^{t/2}$ and $\mu_{s,\theta}(B^{4d}; A) \le 2\eta^{t/2}$.
\end{itemize}
Define $p_i = c 2^{-i/(80d)}$ where $c$ is a positive constant defined so that
$\sum_{i \in [k]} p_i = 1$. 
Then 
\[
	c 
	= \frac{1}{\sum_{i \in [k]} 2^{-i/(80d)}}
	\ge \frac{1}{\sum_{i =0}^{\infty} 2^{-i/(80d)}}
	= 1 - 2^{-1/(80d)}
	\ge 1 - \left(1 - \frac{\ln 2}{80d} + \frac{(\ln 2)^2}{2(80d)^2}\right)
	\ge \frac{1}{160d}.
\]
Thus, for all $i \in [k]$, we have
$p_i \ge p_{k} = c2^{-k/(80d)} \ge c2^{-\log_2 n/(80d)} = c n^{-1/(80d)} \ge \frac{1}{160d}n^{-1/(80d)}$.

We can apply Lemma~\ref{lem:breakdownG} to the pair of sets $(A,B)$ with $r_{\ref{lem:breakdownG}}=2$, $s_{\ref{lem:breakdownG}}=s, d_{\ref{lem:breakdownG}} = 4d, \varepsilon_{\ref{lem:breakdownG}}=\frac{\theta}{m}=\frac{1}{8}\eta \alpha^{4d+t}, (\varepsilon')_{\ref{lem:breakdownG}}=16\eta^{t/2}$, and $p_{i}$ as above since $s \ge 16d$, $p_{i} \ge (\frac{m}{2})^{-1/(40d)}$, and $|A|,|B| \ge \frac{1}{16}\alpha^{t}m \ge 2^{200}(16\eta^{t/2})^{-80}(\frac{1}{8}\eta\alpha^{4d+t})^{-80d}r^{160d}(4d)^{40}k^{40}s^{8}$, given that $c$ is large enough (recall that $m \ge \alpha^{-cd^2}$).
This provides partitions $A = \bigcup_{i \in [k]} A_i$
and $B = \bigcup_{i \in [k]} B_i$ satisfying the following conditions for $\theta_i = \frac{1}{4}p_i\theta$:
\begin{itemize}
  \setlength{\itemsep}{1pt} \setlength{\parskip}{0pt}
  \setlength{\parsep}{0pt}
\item[(i)] for all $i \in [k]$, $|A_i| \le p_i |A|$ and $|B_i| \le p_i|B|$, 
\item[(ii)] for all $(j, i_1, \cdots, i_{4d}) \in [k]^{4d+1}$, $\mu_{s,\theta_j}\left(\prod_{a \in [4d]} B_{i_a}; A_j\right) \le \max\{16\eta^{t/2}, 8\mu_{s,\theta}(B^{4d}; A)\} \le 16\eta^{t/2}$, and
\item[(iii)] for all $(j, i_1, \cdots, i_{4d}) \in [k]^{4d+1}$, $\mu_{s,\theta_j}\left(\prod_{a \in [4d]} A_{i_a}; B_j\right) \le \max\{16\eta^{t/2}, 8\mu_{s,\theta}(A^{4d}; B)\} \le 16\eta^{t/2}$.
\end{itemize}
We now apply the embedding scheme described in Section~\ref{sec:embedding}.
For each $i \in [k]$, we will map $W_{i}^{(1)}$ to $A_i$ and $W_i^{(2)}$ to $B_i$.
We proceed by choosing one pair $(i,j) \in [k] \times [2]$ at a time, following 
the lexicographical ordering on $[k] \times [2]$ in the reverse order.
For each set $W_i^{(j)}$, its corresponding defect parameter used in the embedding scheme is $\theta_i = \frac{1}{4}p_i \theta$,
and therefore 
\[
\gamma \le \max_{i \in [k]} \frac{\max\{|A_i|, |B_i|\}}{\theta_i} \le \max_{i \in [k]} \frac{p_i m/2}{(p_i \theta/4)} = \frac{2m}{\theta} \le \frac{16}{\eta\alpha^{4d+t}}.
\]
By the properties above, the maximum average $s$-th moment defect $\mu_s$ is at most $16\eta^{t/2}$. 
For each $(i,j) \in [k] \times [2]$, we have 
\[
	\frac{|W_{i}^{(j)}|}{\theta_i} 
	\le \frac{2^{-i+1}n}{\frac{1}{4}c 2^{-i/(80d)} \theta} 
	= \frac{2^{-i+1}n}{\frac{1}{4}c 2^{-i/(80d)} (\eta \alpha^{4d+t}m)/8} 
	\le \frac{10240 d n}{2^{i/2} \eta\alpha^{4d+t}m}
	\le \frac{d 2^{21} n}{2^{i/2}\alpha^{72d} m}
	= \frac{1}{2^{i/2}}.
\]
Therefore,
\begin{align*}
	2^{2(4d)+2}\sum_{i \in [k]} \sum_{j  \in [2]} \frac{|W_i^{(j)}|}{\theta_i} \gamma^{2(4d)} \cdot \mu_{4(4d)}
	\le&\, 2^{8d+3} \sum_{i \in [k]} \frac{1}{2^{i/2}} \left(\frac{16}{\eta\alpha^{4d+t}} \right)^{8d} \cdot 16\eta^{t/2} \\
	=&\, 2^{8d+3}\sum_{i \in [k]} \frac{1}{2^{i/2}} \cdot 2^{32d+4} \eta^{t/2 - 8d} \alpha^{-8d(4d+t)} \\
	\le&\, 2^{40d+9} \cdot \eta^{8d} \alpha^{-288d^2} \\
	\le&\, 2^{40d+9} \cdot (\alpha^{36d}2^{-7})^{8d} \alpha^{-288d^2}
	< 2^{-16d+9} < 1.
\end{align*}
Hence, by Theorem~\ref{thm:success}, we can find a copy of $H$ in $G$.
\end{proof}

\subsection{General graphs} \label{subsec:degenerate}

For general graphs, we prove the lemma corresponding to Lemma~\ref{lem:drc_both_prep} in two steps.
In the first step we find sets $A_1, A_2, \cdots, A_r$ that have small average defect in 
one fixed direction. 

\begin{lem} \label{lem:drc_nonbip}
Let $d, s, t, r$ be non-negative integers satisfying $t \ge s$ and $\eta$ be a positive real number.
In every edge two-coloring of the complete graph $K_n$ with red and blue,
in the red graph or the blue graph, there exist sets $A_1 \subseteq \cdots \subseteq A_r$
satisfying the following conditions:
\begin{itemize}
  \setlength{\itemsep}{1pt} \setlength{\parskip}{0pt}
  \setlength{\parsep}{0pt}
\item[(i)] $|A_j| \ge 2^{-2(t+1)(r-1)} n$ for all $j \in [r]$, and
\item[(ii)] for all $\theta \le \eta 2^{-d-2(t+1)(r-1)}n$, $\mu_{s,\theta}(A_{j}^{d}; A_{j+1}) \le 2\eta^{t}$ for all $j \le r-1$.
\end{itemize}
\end{lem}
\begin{proof}
Define $A_{2(r-1)} = V(G)$ and arbitrarily color it with one of the colors. 
For $i \in [2(r-1)]$, suppose that we constructed a set $A_{i}$ of size at least
$2^{-(t+1)(2r-2-i)}n$.
There exists a color, say $c_i$, of density at least $\frac{1}{2}$
in the set $A_{i}$. 
Since $\theta \le \eta 2^{-d}|A_{i}|$, 
we may apply Lemma~\ref{lem:drc}
to the subgraph induced on $A_{i}$ 
with the edges of color $c_i$ to find a set
$A_{i-1} \subseteq A_{i}$ of size $|A_{i-1}| \ge 2^{-t-1}|A_{i}|$
such that $\mu_{s,\theta}(A_{i-1}^{d}; A_{i}) \le 2\eta^{t}$ in the graph consisting 
of the edges of color $c_i$. Color the set $A_{i-1}$ with the color that we used.

Repeat the process to find sets $A_{2r-2} \supseteq A_{2r-3} \supseteq \cdots \supseteq A_{0}$.
Note that $|A_{0}| \ge 2^{-(t+1)(2r-2)}n$.
By the pigeonhole principle, we can find $r$ indices
$i_1 < \cdots < i_r$ for which $A_{i_j}$ are all colored by the same color, say red. 
These sets satisfy Property (i). 
Since $\mu_{s,\theta}(\mathcal{Q}; X) \le \mu_{s,\theta}(\mathcal{Q};Y)$ holds
for all sets $X \supseteq Y$, 
we have 
$\mu_{s,\theta}(A_{i_{a-1}}^d; A_{i_a}) 
	\le \mu_{s,\theta}(A_{i_{a-1}}^d; A_{i_{a-1}+1}) 
	\le 2\eta^t$
in the red graph for all $a \in [2,r]$. 
Thus, the sets $A_{i_1}, A_{i_2}, \cdots, A_{i_r}$ satisfy the claimed properties.
\end{proof}

Given the sets $A_1, \cdots, A_r$ constructed in the previous lemma, we run 
$r$ more rounds of dependent random choice to produce an $r$-tuple of sets that
have small average defect towards each other. At the $i$-th round, we will choose a 
random $t_i$-tuple ${\bf X}_i \in A_i^{t_i}$ and update each set $A_j$ for $j \neq i$
to $N({\bf X}) \cap A_j$. This will ensure that
the average defect of $A_{-i}^{d}$ into $A_{i}$ is small.
The next lemma shows that all the conditions are maintained throughout this process.

\begin{lem} \label{lem:drc_final}
Let $d, r, s$, and $t$ be non-negative integers
satisfying $t \ge s$. Let $\xi$ and $\theta$ be positive real numbers
satisfying $\xi \le 2^{-20(d+t) \cdot 8^{r+1} r}$ and $\theta \le \xi^2 n$. 
In every edge two-coloring of the complete graph $K_n$ with red and blue,
in the red graph or the blue graph, 
there exist sets $A_j$ for $j \in [r]$ satisfying the following properties:
\begin{itemize}
  \setlength{\itemsep}{1pt} \setlength{\parskip}{0pt}
  \setlength{\parsep}{0pt}
\item[(i)] $|A_j| \ge \theta$ for all $j \in [r]$, and
\item[(ii)] $\mu_{s,\theta}(A_{-j}^d; A_j) \le \xi^{t}$ for all $j \in [r]$.
\end{itemize}
\end{lem}
\begin{proof}
For $i=0,1,\cdots, r$, define $t_i = 8^{r+1-i}(d+t)$ and 
$d_i = d + \sum_{j = i+1}^{r} t_j$.
Note that $\frac{1}{6}t_i \ge d_i \ge t_{i+1}$ for all $i \in [r-1]$.
Define $\xi_0 = 2^{-20t_0 r}$ and note that $\xi \le \xi_0$. Define $\theta_0 = \xi n$, $\theta_1 = \xi \theta_0 = \xi^2 n$ and note that it suffices to prove the statement
for $\theta = \theta_1$.

We may apply Lemma~\ref{lem:drc_nonbip} with $d_{\ref{lem:drc_nonbip}} = d_0$, 
$s_{\ref{lem:drc_nonbip}} = 0$, $t_{\ref{lem:drc_nonbip}} = t_0$, $r_{\ref{lem:drc_nonbip}} = r$,
and $\eta_{\ref{lem:drc_nonbip}} = 2^{-16t_0r}$ since
$\theta_1 \le \eta 2^{-d_0 - (t_0+1)(2r-2)}n$.
This gives sets $B_1 \subseteq \cdots \subseteq B_r$
in, without loss of generality, the red graph, satisfying 
$|B_j| \ge 2^{-4t_0 r} n$ for all $j \in [r]$, and
$\mu_{0,\theta_0}(B_{j}^{d_0}; B_{j+1}) \le 2(2^{-16t_0 r})^{t_0} \le \xi^{t_0/2}$ for all $j \in [r-1]$.
Let $A_{0,j} = B_j$ for all $j \in [r]$. 
For each $i = 0,1,\cdots, r$, we will iteratively construct sets $\{A_{i,j}\}_{j \in [r]}$
satisfying the following properties:
\begin{itemize}
  \setlength{\itemsep}{1pt} \setlength{\parskip}{0pt}
  \setlength{\parsep}{0pt}
\item[(a)] $|A_{i,j}| \ge \theta_0$ for all $j \in [i,r]$,
\item[(b)] $|A_{i,j}| \ge \theta_1$ for all $j < i$,
\item[(c)] $\mu_{0,\theta_0}(A_{i,j}^{d_i}; A_{i,{j+1}}) \le \xi^{t_i/2}$ for all $j \in [i,r-1]$,
\item[(d)] $\mu_{s,\theta_1}(A_{i,-j}^{d_i}; A_{i,j}) \le \xi^{t_i/2}$ for all $j 
\le i$,
\item[(e)] $A_{i,j} \subseteq A_{i,j+1}$ for all $j > i$.
\end{itemize}
Note that the properties hold for $i=0$ (the only relevant properties are (a),(c) and (e)).
In the end, the sets $A_j = A_{r,j}$ for $j \in [r]$ satisfy Properties (i) and (ii) by
Properties (b) and (d) since $d_r = d$, $\theta_1 = \xi^2 n$ and $t_r = 8(d+t) \ge 2t$. 

Suppose that for some $i \ge 1$, we have constructed sets 
$\{A_{i-1,j}\}_{j \in [r]}$ for which the properties hold. 
For simplicity, we abuse notation and write $B_j = A_{i-1,j}$ for all $j \in [r]$. 
Let ${\bf X} \in B_i^{t_i}$ be a $t_i$-tuple chosen uniformly at random.
For each $j \neq i$, define ${\bf A}_j = B_j \cap N({\bf X})$ and 
define ${\bf A}_i = B_i$. 
There are several events that we consider. 

\medskip

\noindent \textbf{Event 1}. $|{\bf A}_j| \ge \theta_0$ for all $j \ge i$ and $|{\bf A}_j| \ge \theta_1$ for all $j < i$. 

Note that the claim trivially holds for ${\bf A}_i$ since $|{\bf A}_i| = |B_i| \ge \theta_0$. For $j > i$, since $B_i \subseteq B_{j-1}$,
\begin{align*}
	\BFP\left( |{\bf A}_j| < \theta_0 \right)
	= \mu_{0,\theta_0}(B_i^{t_i}; B_j)
	\le&\, \left(\frac{|B_{j-1}|^{t_i}}{|B_i|^{t_i}}\right)\mu_{0,\theta_0}(B_{j-1}^{t_i}; B_j) \\
	\le &\, \left(\frac{n}{\xi n}\right)^{t_i} \mu_{0,\theta_0}(B_{j-1}^{t_i}; B_j)
	= \xi^{-t_i}\mu_{0,\theta_0}(B_{j-1}^{t_i}; B_j). 
\end{align*}
Since $d_{i-1} \ge t_i$, by Proposition~\ref{prop:productsets} and Property (c), we have
$\mu_{0,\theta_0}(B_{j-1}^{t_i}; B_{j}) \le \mu_{0, \theta_0}(B_{j-1}^{d_{i-1}}; B_j) \le \xi^{t_{i-1}/2}$.
Hence, the probability that $|{\bf A}_j| < \theta_0$ is at most $\xi^{-t_i+t_{i-1}/2} \le \xi^{t_{i-1}/4}$. 
Similarly, for $j < i$, since $B_{i} \subseteq B_{-j}$,
\begin{align*}
	\BFP\left(|{\bf A}_j| < \theta_1 \right)
	= \mu_{0,\theta_1}(B_i^{t_i}; B_j)
	\le&\, \left(\frac{|B_{-j}|^{t_i}}{|B_i|^{t_i}}\right)\mu_{0,\theta_1}(B_{-j}^{t_i}; B_j) \\
	\le &\, \left(\frac{n}{\xi^2 n}\right)^{t_i} \mu_{0,\theta_1}(B_{-j}^{t_i}; B_j)
	\le \xi^{-2t_i} \mu_{s,\theta_1}(B_{-j}^{d_{i-1}}; B_j). 
\end{align*}
By Property (d), this is at most $\xi^{-2t_i+t_{i-1}/2} \le \xi^{t_{i-1}/4}$.
Since there are $r$ total events, the probability of Event 1 is at least $1 - r\xi^{t_{i-1}/4} > \frac{3}{4}$. 

\medskip

\noindent \textbf{Event 2}. $\mu_{0,\theta_0}(B_{j}^{d_i}; {\bf A}_{{j+1}}) \le 4r\xi^{t_{i-1}/4}$ for all $j \in [i,r-1]$.

For $j \in [i,r-1]$, since $B_i \subseteq B_j$, we have
\begin{align*}
	\BBE[\mu_{0, \theta_0}(B_{j}^{d_i}; {\bf A}_{j+1})]
	=&\, \frac{1}{|B_i|^{t_i}} \sum_{Q \in B_{i}^{t_i}} \mu_{0, \theta_0}(B_{j}^{d_i}; N(Q) \cap B_{j+1}) \\
	\le&\, \left(\frac{|B_j|}{|B_i|}\right)^{t_i}\frac{1}{|B_j|^{t_i}} \sum_{Q \in B_{j}^{t_i}} \mu_{0, \theta_0}(B_{j}^{d_i}; N(Q) \cap B_{j+1})
	\le \xi^{-2t_i} \BBE_{\nu}[\mu_{0,\theta_0}(B_j^{d_i}; {\bf A}_{j+1})],
\end{align*}
where $\nu$ is the probability distribution obtained by choosing ${\bf X}$ 
uniformly in the set $B_j^{t_i}$ instead of $B_i^{t_i}$.
By Proposition~\ref{prop:defect_transfer} and Property (c), 
$\BBE_{\nu}[\mu_{0,\theta_0}(B_j^{d_i}; {\bf A}_{j+1})]
	\le \mu_{0, \theta_0}(B_{j}^{d_{i-1}}; B_{j+1}) \le \xi^{t_{i-1}/2}$.
Since $\xi^{-2t_i}\cdot \xi^{t_{i-1}/2} \le \xi^{t_{i-1}/4}$,
by Markov's inequality, the probability of Event 2 for a fixed index $j$ is less than $1 - \frac{1}{4r}$.
Since there are at most $r$ choices for the index $j$, the probability of Event 2 is greater than $\frac{3}{4}$. 

\medskip

\noindent \textbf{Event 3}. $\mu_{s,\theta_1}(B_{-j}^{d_i}; {\bf A}_{j}) \le 4r \xi^{t_{i-1}/4}$ for all $j < i$ . 

As in Event 2, we have
$\BBE[\mu_{s, \theta_1}(B_{-j}^{d_i}; {\bf A}_{j})] \le \xi^{-2t_i} \mu_{s, \theta_1}(B_{-j}^{d_{i-1}}; B_{j}) \le \xi^{t_{i-1}/4}$ (by Property (d)).
Therefore, by Markov's inequality, the probability of Event 3 for a fixed $j$ is less than $1 - \frac{1}{4r}$. 
Since there are at most $r$ choices for $j$, the probability of Event 3 is greater than $\frac{3}{4}$. 

\medskip

\noindent \textbf{Event 4}. $\sum_{Q \in {\bf A}_{-i}^{d_i}} \omega_{\theta_1}(Q; B_i)^{s} \le 4|B_{-i}|^{d_i} \xi^{t_i}$.

Note that
$\BBE\left[\sum_{Q \in {\bf A}_{-i}^{d_i}} \omega_{\theta_1}(Q; B_i)^{s} \right]
	\le \sum_{Q \in B_{-i}^{d_i}} \omega_{\theta_1}(Q; B_i)^s \cdot \BFP\left(Q \in {\bf A}_{-i}^{d_i}\right)$.
Since $\theta_1 = \xi \theta_0 \le \xi |B_i|$, if $\omega_{\theta_1}(Q; B_i) \neq 0$, then
$|N(Q; B_i)| = \frac{\theta_1}{\omega_{\theta_1}(Q; B_i)} \le \frac{\xi|B_i|}{\omega_{\theta_1}(Q; B_i)}$ and thus $\BFP(Q \in {\bf A}_{-i}^{d_i}) = \left(\frac{|N(Q;B_i)|}{|B_i|}\right)^{t_i} \le \left(\frac{\xi}{\omega_{\theta_1}(Q; B_i)}\right)^{t_i} \le \frac{\xi^{t_i}}{\omega_{\theta_1}(Q; B_i)^s}$.
Therefore, $\BBE\left[\sum_{Q \in {\bf A}_{-i}^{d_i}} \omega_{\theta_1}(Q; B_i)^{s} \right] 
\le |B_{-i}|^{d_i} \xi^{t_i}$.
Hence, with probability greater than $\frac{3}{4}$, we have 
$\sum_{Q \in {\bf A}_{-i}^{d_i}} \omega_{\theta_1}(Q; B_i)^{s} \le 4|B_{-i}|^{d_i} \xi^{t_i}$.

\medskip

Therefore, with positive probability all four events hold.
Let $X$ be a particular choice of ${\bf X}$ for which all four events hold, and define $A_{i,j} = {\bf A}_j$ 
for this choice of $X$. 
Event~1 immediately implies Properties (a) and (b).
Since $|A_{i,j}| \ge \theta_1 \ge \xi^2 |B_j|$ holds for all $j \in [r]$, it follows by
Event~2 that for all $j \in [i,r-1]$,
\[
	\mu_{0,\theta_0}(A_{i,j}^{d_i}; A_{i,j+1})
	\le \xi^{-2d_i} \mu_{0,\theta_0}(B_{j}^{d_i}; A_{i,j+1})
	\le \xi^{-2d_i} \cdot 4r \xi^{t_{i-1}/4}
	\le \xi^{t_i/2},
\]
implying Property (c).
Similarly, Event 3 implies $\mu_{s,\theta_1}(A_{i,-j}^{d_i}; A_j) \le \xi^{t_i/2}$ for all $j < i$.
Finally, Event 4 implies 
\[
	\frac{1}{|A_{i,-i}|^{d_i}} \sum_{Q \in A_{i,-i}^{d_i}} \omega_{\theta_1}(Q; A_{i,i})^{s} 
	\le \frac{\xi^{-2d_i}}{|B_{-i}|^{d_i}}\sum_{Q \in B_{-i}^{d_i}} \omega_{\theta_1}(Q; B_i)^{s} 
	\le \xi^{-2d_i} \cdot 4\xi^{t_i} \le \xi^{t_i/2},
\]
proving Property (d). Property (e) follows from the same property the previous step, and the definition $A_{i,j} = N(X) \cap A_{i-1,j}$ for $j \neq i$.
\end{proof}

We conclude this section with the proof of our main theorem, Theorem~\ref{thm:main}.

\vspace{0.2cm}
\noindent \textbf{Theorem~\ref{thm:main}}.
There exists a constant $c$ such that the following holds for every natural number $d$, $r$, and $n$ satisfying $n \ge 2^{d^2 2^{cr}}$.
For every edge two-coloring of the complete graph on at least $2^{d2^{cr}}n$ vertices, one of 
the colors is universal for the family of 
$d$-degenerate $r$-colorable graphs on at most $n$ vertices.
\vspace{-0.2cm}
\begin{proof}
Define $m = 2^{d 2^{cr}}n$ for a large enough constant $c$. 
Suppose that we are given an edge coloring of $K_m$ with two colors red and blue.
Define $t = s = 32d$ and $\xi = 2^{-d 2^{(c/2)r}}$.  Define $\theta = \xi^2 m$.
Lemma~\ref{lem:drc_final} with $d_{\ref{lem:drc_final}} = 4d$ implies that
in the red graph or the blue graph, there exist sets $\{A_j\}_{j \in [r]}$ satisfying the following properties:
\begin{itemize}
  \setlength{\itemsep}{1pt} \setlength{\parskip}{0pt}
  \setlength{\parsep}{0pt}
\item[(i)] $|A_j| \ge \theta$ for all $j \in [r]$, and
\item[(ii)] $\mu_{s,\theta}(A_{-j}^{4d}; A_j) \le \xi^{t}$.
\end{itemize}
Define $G$ as the subgraph of $K_m$ consisting of the edges of the color realizing the two properties above.

Let $H$ be a $d$-degenerate $r$-colorable graph on at most $n$ vertices. 
By Lemma~\ref{lem:split}, there exists a vertex partition $V(H) = \bigcup_{(i,j) \in [k] \times [r]} W_i^{(j)}$
satisfying the following properties:
\begin{itemize}
  \setlength{\itemsep}{1pt} \setlength{\parskip}{0pt}
  \setlength{\parsep}{0pt}
\item[(i)] $k \le \log_{2} n$,
\item[(ii)] for all $(i,j) \in [k] \times [r]$, we have $|W_{i}^{(j)}| \le 2^{-i+1}n$,
\item[(iii)] for all $j \in [r]$, the set $\bigcup_{i \in [k]} W_{i}^{(j)}$ is an independent set, and 
\item[(iv)] for all $(i,j) \in [k] \times [r]$, each vertex $v \in W_i^{(j)}$ has at most $4d$ neighbors in
$\bigcup_{i'\ge i, j' \in [r]} W_{i'}^{(j')}$.
\end{itemize}

Define $p_i = c' 2^{-i/(80d)}$ where $c'$ is a positive constant defined so that
$\sum_{i \in [k]} p_i = 1$. 
Then 
\[
	c' 
	= \frac{1}{\sum_{i \in [k]} 2^{-i/(80d)}}
	\ge \frac{1}{\sum_{i =0}^{\infty} 2^{-i/(80d)}}
	= 1 - 2^{-1/(80d)}
	\ge 1 - \left(1 - \frac{\ln 2}{80d} + \frac{(\ln 2)^2}{2(80d)^2}\right)
	\ge \frac{1}{160d}.
\]
Thus, for all $i \in [k]$, we have
$p_i \ge p_{k} = c'2^{-k/(80d)} \ge c'2^{-\log_2 n/(80d)} = c' n^{-1/(80d)} \ge \frac{1}{160d}n^{-1/(80d)}$.
Since $p_i \ge n^{-1/(40d)}$, $s \ge 16d$, and 
$m > n \ge 2^{d^22^{cr}}$, if $c$ is large enough, 
then we can apply Lemma~\ref{lem:breakdownG} to $G$
with $d_{\ref{lem:breakdownG}} = 4d$,
$\varepsilon_{\ref{lem:breakdownG}} = \xi^2$,
$(\varepsilon')_{\ref{lem:breakdownG}} = \xi^{t/2}$ and 
sets $\{A_j\}_{j \in [r]}$, to obtain 
sets $\{V_i^{(j)}\}_{(i,j) \in [k] \times [r]}$ satisfying the following conditions for 
$\theta_i = \frac{1}{2r}p_i\theta$:
\begin{itemize}
  \setlength{\itemsep}{1pt} \setlength{\parskip}{0pt}
  \setlength{\parsep}{0pt}
\item[(i)] for all $(i,j) \in [k] \times [r]$, $|V_i^{(j)}| \le p_i |A_j|$, and
\item[(ii)] for all $((i,j), (i_1,j_1), \cdots, (i_{4d}, j_{4d})) \in ([k]\times[r])^{4d+1}$ satisfying $j_1, \cdots, j_{4d} \neq j$,
\[
	\mu_{s,\theta_i}\left(\prod_{a \in [4d]} V_{i_a}^{(j_a)}; V_i^{(j)}\right) \le \max \left\{\xi^{t/2}, 8r^{4d}\xi^{-4d} \mu_{s,\theta}(A_{-j}^{4d}; A_j)\right\} \le \xi^{t/2}.
\]
\end{itemize}
We now apply the embedding scheme described in Section~\ref{sec:embedding}.
We choose one pair $(i,j) \in [k] \times [r]$ at a time, 
following the lexicographical ordering of $[k] \times [r]$ in the reverse order, 
and map $W_{i}^{(j)}$ to $V_{i}^{(j)}$. 
For each set $W_{i}^{(j)}$, its corresponding defect parameter used in the embedding scheme is
$\theta_{i} = \frac{1}{2r}p_i \theta$, and therefore,
\[
	\gamma \le \max_{i \in [k]} \frac{\max_{j \in [r]} |V_i^{(j)}|}{\theta_i} \le \max_{i \in [k]} \frac{p_i m}{p_i \xi^2 m/2r} = \frac{2r}{\xi^2}.
\]  
Moreover, by Property (ii) above, the maximum average $s$-th moment defect 
satisfies $\mu_s \le \xi^{t/2}$.
For all $(i,j) \in [k] \times [r]$, since $m \ge 2^{d 2^{cr}}n \ge 640rd\xi^{-3}n$ for large enough $c$,
\[
	\frac{|W^{(j)}_i|}{\theta_i} = \frac{2^{-i+1}n}{\frac{1}{2r} c' 2^{-i/(80d)} \theta} 
	\le \frac{640 rd n}{2^{i/2} \cdot \xi^2 m}
	\le \frac{1}{2^{i/2}} \cdot \xi.
\]	
Therefore, for large enough $c$,
\begin{align*}
	2^{8d+2}\gamma^{8d} \mu_{16d} \sum_{i \in [k]} \sum_{j  \in [r]} \frac{|W_i^{(j)}|}{\theta_i} 
	\le&\, 2^{8d+2} \left(\frac{2r}{\xi^2}\right)^{8d}\xi^{t/2} \sum_{i \in [k]} \frac{r\xi}{2^{i/2}} \\
	\le&\, 2^{16d+4} r^{8d+1} \xi < 1.
\end{align*}
Hence, by Theorem~\ref{thm:success}, we can find a monochromatic copy of $H$.
\end{proof}

\section{Random pruning} \label{sec:pruning}

In this section, we prove Lemma~\ref{lem:breakdownG}, the final ingredient of the proof.
One of our main tools is the following concentration inequality (see \cite[Theorem 3.1]{McDiarmid}).

\begin{thm} \label{thm:concentration}
Let ${\bf X} = (X_1, X_2, \cdots, X_n)$ be a family of independent random variables with $X_i$
taking values in a set $\Omega_i$ for each $i$. Suppose that the real-valued function $f$
defined on $\prod_{i \in [n]} \Omega_i$ satisfies $|f(\vec{x}) - f(\vec{y})| \le c_i$ whenever
the vectors $\vec{x}$ and $\vec{y}$ differ only in the $i$-th coordinate. Then
\[
	\BFP\left(\Big|f({\bf X}) - \BBE[f({\bf X})] \Big| \ge t\right)
	\le 2 e^{-2t^2 / \sum_{i \in [n]} c_i^2}.
\]
\end{thm}

The next lemma shows that given the collection of sets obtained in 
Lemmas~\ref{lem:drc_both_prep} and \ref{lem:drc_nonbip},
we can further impose that the defect is not concentrated too much on individual vertices.
This additional condition will help us later when taking a random partition.

\begin{lem} \label{lem:add_mindeg}
Let $\varepsilon$ be a fixed positive real number,
and $d, s, r$ be fixed natural numbers satisfying $s \ge 4d$.
Let $m$ be a natural number satisfying $m \ge 2^{24}s^8d^8r^{18}\varepsilon^{-8}$, and
$\theta$ be a natural number satisfying $\theta \ge \varepsilon m$.
Let $A_1, A_2, \cdots, A_r$ be (not necessarily disjoint) vertex subsets satisfying
$\varepsilon m \le |A_i| \le m$ and $\mu_{s,\theta}(A_{-i}^d; A_i) < 1$ for all $i \in [r]$. Then there exist subsets
$B_i \subseteq A_i$ for all $i \in [r]$ satisfying the following properties:
\begin{itemize}
  \setlength{\itemsep}{1pt} \setlength{\parskip}{0pt}
  \setlength{\parsep}{0pt}
\item[(i)] $|B_i| \ge 2^{-1/(2d)}|A_i|$ for all $i \in [r]$,
\item[(ii)] $\mu_{s,\theta}(B_{-i}^d; B_i) \le 2\mu_{s,\theta}(A_{-i}^d; A_i)$ for all $i \in [r]$, and
\item[(iii)] for every $i \in [r]$ and $v \in B_{-i}$, $\sum_{Q:v \in Q \in B_{-i}^d} \omega_{\theta}(Q; B_i)^s \le 2|B_{-i}|^{d-5/8}$.
\end{itemize}
\end{lem}
\begin{proof}
We will fix $\theta$ throughout the proof and thus for simplicity will use the notation
$\omega_{\theta} = \omega$ and $\mu_{s,\theta} = \mu_{s}$.
Let $R_i \subseteq A_{-i}$ be the set of vertices $v \in A_{-i}$ such that 
$\sum_{Q : v \in Q \in A_{-i}^d} \omega(Q; A_i)^s \ge |A_{-i}|^{d-5/8}$. Then
\[
	|R_i| \cdot |A_{-i}|^{d-5/8} 
	\le \sum_{v \in R_i} \sum_{Q : v \in Q \in A_{-i}^d} \omega(Q; A_i)^s 
	\le d \cdot |A_{-i}|^d \mu_{s}(A_{-i}^d; A_i)
	< d|A_{-i}|^{d}.
\]
Therefore, $|R_i| < d|A_{-i}|^{5/8} \le drm^{5/8}$. 
For each $i \in [r]$, let $B_i$ be the set obtained from $A_i$ by removing the 
vertices in $\bigcup_{j \in [r]} R_j$. 
Then $|B_i| \ge |A_i| - dr^2 m^{5/8} \ge 2^{-1/(2d)}|A_i|$ and (i) holds.

Fix $i \in [r]$. Since $s \ge 4d$, by Proposition~\ref{prop:mindeg}, all $d$-tuples in 
$A_{-i}^d$ have at least $\frac{\theta}{|A_{-i}|^{1/4}} \ge \frac{\varepsilon m}{(rm)^{1/4}} \ge \varepsilon r^{-1/4}m^{3/4}$ common neighbors in $A_i$.
Since $B_i$ is obtained from $A_i$ by removing at most $dr^2 m^{5/8}$ vertices, 
all $d$-tuples $Q \in B_{-i}^d$ satisfy 
\begin{align*}
	|N(Q; B_i)| 
	\ge&\, \left(1 - \frac{dr^2m^{5/8}}{|N(Q;A_i)|}\right)|N(Q;A_i)|
	\ge \left(1 - \frac{dr^2m^{5/8}}{\varepsilon r^{-1/4}m^{3/4}}\right)|N(Q;A_i)| \\
	\ge&\, \left(1 - \frac{dr^{9/4}}{\varepsilon m^{1/8}}\right)|N(Q;A_i)|
	\ge 2^{-1/(2s)}|N(Q;A_i)|.
\end{align*}
Hence, $\omega(Q; B_i) \le 2^{1/(2s)}\omega(Q; A_i)$. Therefore,
\begin{align*}
	\mu_{s}(B_{-i}^d; B_i) 
	=&\, \frac{1}{|B_{-i}|^d} \sum_{Q \in B_{-i}^d} \omega(Q; B_i)^s
	\le \frac{1}{|B_{-i}|^d} \sum_{Q \in A_{-i}^d} 2^{1/2}\omega(Q; A_i)^s \\
	=&\, 2^{1/2}\frac{|A_{-i}|^d}{|B_{-i}|^d} \left( \frac{1}{|A_{-i}|^d}\sum_{Q \in A_{-i}^d} \omega(Q; A_i)^s \right)
	\le 2\mu_{s}(A_{-i}^d; A_i).
\end{align*}
This proves (ii). 
Moreover, since $B_i = A_i \setminus R_i$, each vertex $v \in B_i$ satisfies 
$\sum_{Q : v \in Q \in A_{-i}^d} \omega(Q; A_i)^s < |A_{-i}|^{d-5/8}$. 
Therefore, 
\[
	\sum_{Q : v \in Q \in B_{-i}^d} \omega(Q; B_i)^s
	\le \sum_{Q : v \in Q \in B_{-i}^d} 2^{1/2}\omega(Q; A_i)^s
	\le 2^{1/2}|A_{-i}|^{d-5/8}
	\le 2|B_{-i}|^{d-5/8},
\]
proving (iii).
\end{proof}

Lemma~\ref{lem:add_mindeg} prepares an $r$-tuple of sets with small average defect towards each 
other by further imposing that the defect is not concentrated too much on individual vertices. 
We now prove Lemma~\ref{lem:breakdownG} by taking a random partition and showing that the 
defect is well-distributed. 

\vspace{0.2cm}
\noindent \textbf{Lemma \ref{lem:breakdownG}}. 
Let $k, d, s, r$ be fixed natural numbers satisfying $r \ge 2$, $s \ge 4d$,
and $\varepsilon, \varepsilon'$ be fixed positive real numbers.
Let $m$ be a natural number satisfying $m \ge 2^{200}(\varepsilon')^{-80}\varepsilon^{-80d}r^{160d}d^{40}k^{40}s^{8}$, and $\theta$ be a natural number satisfying $\theta \ge \varepsilon m$.
Let $p_i$ for $i \in [k]$ be positive real numbers  
satisfying $\sum_{i \in [k]} p_i \le 1$ and $p_i \ge m^{-1/(10d)}$ for all $i \in [k]$.
Suppose that $\{A_j\}_{j \in [r]}$ are vertex subsets 
with $\varepsilon m \le |A_j| \le m$, 
satisfying $\mu_{s, \theta}(A_{-j}^{d}; A_j) < \frac{1}{2}$ for all $j \in [r]$.
Define $\theta_i = \frac{1}{2r}p_i \theta$ for all $i \in [k]$.
Then there exist sets $\{V_{i}^{(j)}\}_{(i,j) \in [k] \times [r]}$ 
satisfying the following conditions:
\begin{itemize}
  \setlength{\itemsep}{1pt} \setlength{\parskip}{0pt}
  \setlength{\parsep}{0pt}
\item[(i)] For all $i \in [k]$ and $j \in [r]$, we have $V^{(j)}_{i} \subseteq A_j$ and $|V^{(j)}_{i}| \le p_i |A_j|$.
\item[(ii)] For all $((i, j), (i_1, j_1), \cdots, (i_d, j_d)) \in ([k] \times [r])^{d+1}$ satisfying $j_1, \cdots, j_d \neq j$,
\[
	\mu_{s,\theta_{i}}\left(\prod_{a \in [d]} V^{(j_a)}_{i_a}; V^{(j)}_{i}\right) \le \max\left\{\varepsilon', 8r^d\varepsilon^{-d} \mu_{s,\theta}(A_{-j}^d; A_j)\right\}.
\]
Moreover, if $r=2$, then the factor $r^d \varepsilon^{-d}$ can be replaced by $1$.
\item[(iii)] For all $(i,j) \neq (i',j')$ the sets $V_i^{(j)}$ and $V_{i'}^{(j')}$ are disjoint.
\end{itemize}
\vspace{-0.2cm}

\begin{proof}
Note that the given condition on $\{p_i\}_{i\in[k]}$ implies $k \le m^{1/(10d)}$.
By Lemma~\ref{lem:add_mindeg}, we can find subsets $B_j \subseteq A_j$ for $j \in [r]$
satisfying the following conditions:
\begin{itemize}
  \setlength{\itemsep}{1pt} \setlength{\parskip}{0pt}
  \setlength{\parsep}{0pt}
\item[(a)] $|B_j| \ge 2^{-1/(2d)}|A_j|$ for all $j \in [r]$,
\item[(b)] $\mu_{s,\theta}(B_{-j}^d; B_j) \le 2\mu_{s,\theta}(A_{-j}^d; A_j) < 1$ for all $j \in [r]$, and
\item[(c)] for every $j \in [r]$ and $v \in B_{-j}$, we have $\sum_{Q: v \in Q \in B_{-j}^d} \omega_{\theta}(Q; B_j)^s \le 2|B_{-j}|^{d-5/8}$.
\end{itemize}
Define $q_i = \frac{1}{r}p_i$ for $i \in [k]$. 
Color the vertices with $[k] \times [r]$ randomly 
where each vertex receives color $(i,j)$ with probability 
$q_i$ and the outcome for each vertex is independent.
For each $(i,j) \in [k] \times [r]$, let $V^{(j)}_i \subseteq B_j$ be the set of vertices of color $(i,j)$. 
Let $E_{1}$ be the event that for all $(i,j) \in [k] \times [r]$ and all $d$-tuples $Q \in B_{-j}^{d}$, we have
$|N(Q; V^{(j)}_i)| \ge \frac{1}{2}q_i |N(Q; B_j)|$.
Let $E_{2}$ be the event that for all $((i, j), (i_1, j_1), \cdots, (i_d, j_d)) \in ([k] \times [r])^{d+1}$ satisfying $j_1, \cdots, j_d \neq j$,
\[
	\mu_{s,\theta}\left( \prod_{a \in [d]} V^{(j_a)}_{i_a}; B_j \right)
	\le \max\left\{\varepsilon', 4r^d\varepsilon^{-d} \mu_{s,\theta}\left( B_{-j}^{d}; B_j \right) \right\},
\]
where if $r=2$, then we replace the factor $r^d \varepsilon^{-d}$ by $1$. 
Let $E_3$ be the event that for all $(i,j) \in [k] \times [r]$, we have 
$\frac{1}{2^{1/(2d)}}q_i|B_j| \le |V_{i}^{(j)}| \le 2q_i |B_j|$.

We first show that Properties (i), (ii), (iii) follows
conditioned on the events $E_1, E_2$ and $E_3$.
Property (iii) immediately follows from definition of the sets.
Property (i) follows from $E_3$ and $|B_j| \le |A_j|$ for $j \in [r]$.
For some $j \in [r]$, fix a $d$-tuple $Q \in B_{-j}^d$. 
Since $|N(Q; V^{(j)}_i)| \ge \frac{1}{2}q_i |N(Q;B_j)|$ holds by $E_1$,
if $|N(Q; B_j)| \ge \theta$, then we have $|N(Q; V^{(j)}_i)| \ge \frac{1}{2}q_i \theta$.
Therefore, if $\omega_{\theta}(Q;B_j) = 0$, then $\omega_{\theta_i}(Q; V^{(j)}_i) = 0$. 
Otherwise, if $\omega_{\theta}(Q;B_j) \neq 0$, then
$\omega_{\theta}(Q;B_j) = \frac{\theta}{|N(Q;B_j)|} \ge \frac{q_i \theta}{2|N(Q;V^{(j)}_i)|} = \omega_{\theta_i}(Q; V^{(j)}_i)$.
Therefore, $\omega_{\theta_i}(Q; V^{(j)}_i) \le \omega_{\theta}(Q; B_j)$ in both cases.  
This implies that for all $((i, j), (i_1, j_1), \cdots,$ $(i_d, j_d)) \in ([k] \times [r])^{d+1}$,
we have by $E_2$,
\begin{align*}
	\mu_{s,\theta_i}\left( \prod_{a \in [d]} V_{i_a}^{(j_a)}; V^{(j)}_i\right)
	\le \mu_{s,\theta}\left( \prod_{a \in [d]} V^{(j_a)}_{i_a}; B_j \right)
	\le \max\left\{\varepsilon', 4r^d \varepsilon^{-d} \mu_{s,\theta}\left( B_{-j}^{d}; B_j \right) \right\}.
\end{align*}
Therefore, by (b), we have Property (ii) (similarly Property (ii) holds for $r=2$ as well). 

It thus suffices to show that the intersection of $E_1, E_2, E_3$ has non-zero 
probability.
To compute the probability of $E_1$, note that for $(i,j) \in [k] \times [r]$, 
since $s \ge 4d$ and $\mu_{s,\theta}(B_{-j}; B_j) < 1$, by Proposition~\ref{prop:mindeg}, all 
$d$-tuples $Q \in B_{-j}^d$ have at least $\frac{\theta}{|B_{-j}|^{1/4}} \ge \frac{\varepsilon m}{(rm)^{1/4}} = \varepsilon r^{-1/4} m^{3/4}$ common neighbors in $B_j$. 
Hence, for a fixed $d$-tuple $Q \in B_{-j}^d$,
\[
	\BBE\left[ |N(Q; V^{(j)}_i)| \right] 
	= q_i |N(Q; B_j)| 
	\ge q_i \varepsilon r^{-1/4} m^{3/4}.
\]
Therefore, by Theorem~\ref{thm:concentration}, the probability that
$|N(Q; V^{(j)}_i)| < \frac{1}{2}q_i |N(Q;B_i)|$ for a fixed $Q \in B_{-j}^d$ and $(i,j) \in [k] \times [r]$
is at most $2e^{-2(q_i\varepsilon r^{-1/4}m^{3/4})^2/m} = 2e^{-2q_i^2\varepsilon^2r^{-1/2} m^{1/2}}$.
Since there are at most $(rm)^d$ choices for $Q \in B_{-j}^d$,
there are at most $rk \cdot (rm)^d$ such events. 
By the union bound, 
\[
	\BFP(E_1^c) 
	\le r^{d+1}km^d \cdot 2e^{-2q_i^2\varepsilon^2r^{-1/2} m^{1/2}}
	\le r^{d+1}km^d \cdot 2e^{-2 m^{1/4}}
	< \frac{1}{4}.
\]

Note that by Theorem~\ref{thm:concentration},
\[
	\BFP\left( 2^{-1/(2d)} q_i |B_j| \le |V_i^{(j)}| \le 2q_i |B_j| \right)
	\ge 1 - 2 e^{-2(1-2^{-1/(2d)})^2q_i^2|B_j|}
	\ge 1 - 2 e^{-\varepsilon /(64d^2r^2)m^{1/2} }
	> 1 - \frac{1}{3(kr)^{d+1}}.
\]
There are $(kr)^{d+1}$ choices for 
$((i, j), (i_1, j_1), \cdots, (i_d, j_d)) \in ([k] \times [r])^{d+1}$,
and thus $P(E_3) > \frac{3}{4}$. 

To compute the probability of $E_2$, fix
$((i, j), (i_1, j_1), \cdots, (i_d, j_d)) \in ([k] \times [r])^{d+1}$.
To simplify notation, define $\mathcal{Q} = \prod_{a \in [d]} V^{(j_a)}_{i_a}$
and $\mathcal{B} = \prod_{a \in [d]} B_{j_a}$.
Note that
\begin{align} \label{eq:weight_dist}
	\BBE\left[ \sum_{Q \in \mathcal{Q} } \omega_{\theta}(Q; B_j)^s \right]
	= &\, \sum_{Q \in\mathcal{B}} \omega_{\theta}(Q; B_j)^s \cdot \BFP\left(Q \in \mathcal{Q}\right).
\end{align}
Fix $Q \in \mathcal{B}$.
If all vertices in $Q$ are distinct, then $\BFP(Q \in \mathcal{Q}) = \prod_{a \in [d]} q_{i_a}$.
Otherwise, if $\partial \mathcal{B}$ is the set of $d$-tuples in 
$\mathcal{B}$ where not all vertices are distinct, then by Proposition~\ref{prop:boundary},
we have $\sum_{Q \in \partial \mathcal{B}} \omega_{\theta}(Q; B_j)^s 
\le \frac{d(d-1)}{2\varepsilon m}\sum_{Q \in \mathcal{B}} \omega_{\theta}(Q; B_j)^s 
\le \frac{d(d-1)}{2\varepsilon m}\sum_{Q \in B_{-j}^{d}} \omega_{\theta}(Q; B_j)^s< \varepsilon ^{-1}d^2r^dm^{d-1}$, by Property (b) and $|B_{-j}| \le rm$.
Thus, in \eqref{eq:weight_dist},
\begin{align*}
	\BBE\left[ \sum_{Q \in \mathcal{Q} } \omega_{\theta}(Q; B_j)^s \right]
	\le &\, \varepsilon ^{-1}d^2r^dm^{d-1} + \left(\prod_{a \in [d]}q_{i_a}\right) \cdot \sum_{Q \in \mathcal{B}} \omega_{\theta}(Q; B_j)^s \\
	= &\, \varepsilon ^{-1}d^2r^dm^{d-1} + \left(\prod_{a \in [d]}q_{i_a}\right) |\mathcal{B}|\mu_{s,\theta}(\mathcal{B}; B_j).
\end{align*}

For each vertex $v \in B_{-j}$, by (c), we know that
$\sum_{Q : v \in Q \in \mathcal{B}} \omega_{\theta}(Q; B_j)^s \le 2|B_{-j}|^{d-5/8}$.
Therefore, the random variable $\sum_{Q \in \mathcal{Q}} \omega_{\theta}(Q;B_j)^s$
can change by at most $2|B_{-j}|^{d-5/8} \le 2r^{d}m^{d-5/8}$ if we change the outcome of a single vertex.
Since $\prod_{a \in [d]}q_{i_a} \ge r^{-d}m^{-1/10}$,
by Theorem~\ref{thm:concentration}, the probability that
the random variable is greater than 
$\lambda := \left(\prod_{a \in [d]}q_{i_a}\right) |\mathcal{B}| \max\left\{ 2\mu_{s,\theta}(\mathcal{B}; B_j), \frac{\varepsilon'}{2}\right\}$ is at most
$2e^{-2\frac{(\lambda')^2}{m \cdot (2r^dm^{d-5/8})^2}}
= 2e^{-(\lambda')^2 / (2r^{2d}m^{2d-1/4})}$, where
\begin{align*}
	\lambda' := \lambda - \BBE\left[ \sum_{Q \in \mathcal{Q} } \omega_{\theta}(Q; B_j)^s \right]
	\ge&\, \lambda - \max\left\{3\varepsilon^{-1}r^dd^2m^{d-1}, \frac{3}{2} \left(\prod_{a \in [d]}q_{i_a}\right) |\mathcal{B}|\mu_{s,\theta}(\mathcal{B}; B_j)\right\} 
	\ge \frac{1}{4}\lambda.
\end{align*}
Since
\[
	\lambda 
	= \left(\prod_{a \in [d]}q_{i_a}\right) |\mathcal{B}| \max\left\{ 2\mu_{s,\theta}(\mathcal{B}; B_j), \frac{\varepsilon'}{2}\right\}
	\ge \left(r^{-1}m^{-1/(10d)}\right)^{d} (2^{-1/(2d)}\varepsilon m)^{d} \cdot \frac{\varepsilon'}{2}
	= \frac{\varepsilon'\varepsilon^{d}}{4r^d}m^{d-1/10},
\]
the above mentioned probability is at most
\[
	2e^{-(\lambda')^2 / (2r^{2d}m^{2d-1/4})}
	\le 2e^{-\lambda^2 / (32r^{2d}m^{2d-1/4})}
	\le 2e^{-2^{-9}(\varepsilon')^2\varepsilon^{2d}r^{-4d}m^{1/20}}
	< \frac{3}{4(kr)^{d+1}}.
\]

Since there are at most $(kr)^{d+1}$ choices of indices, 
we see that $\sum_{Q \in \mathcal{Q} } \omega_{\theta}(Q; B_j)^s \le \lambda$ 
holds for all choices, with probability at least $\frac{3}{4}$. Recall that $E_3$ holds with probability greater than $\frac{3}{4}$.
We show that $E_2$ holds if both these events hold. 
For a fixed $((i, j), (i_1, j_1), \cdots, (i_d, j_d)) \in ([k] \times [r])^{d+1}$,
following the notation above, event $E_3$ implies that 
$\left(\prod_{a \in [d]}q_{i_a}\right) |\mathcal{B}| \le 2^{1/2}|\mathcal{Q}|$. Therefore,
\begin{align*}
	\sum_{Q \in \mathcal{Q} } \omega_{\theta}(Q; B_j)^s 
	\le&\,  \lambda = \left(\prod_{a \in [d]}q_{i_a}\right) |\mathcal{B}| \max\left\{ 2\mu_{s,\theta}(\mathcal{B}; B_j), \frac{\varepsilon'}{2}\right\} \\
	\le &\, 2^{1/2}|\mathcal{Q}| \max\left\{ 2\mu_{s,\theta}(\mathcal{B}; B_j), \frac{\varepsilon'}{2}\right\}
	\le |\mathcal{Q}| \max\left\{ 2^{3/2}\mu_{s,\theta}(\mathcal{B}; B_j), \varepsilon'\right\}.
\end{align*}

Note that if $r=2$, then $\mathcal{B} = B_{-j}^d$ and thus $\mu_{s,\theta}(\mathcal{B}; B_j) = \mu_{s,\theta}(B_{-j}^{d}; B_j)$. Therefore, we have $E_3$.
If $r \neq 2$, then since all sets $B_a$ for $a \in [r]$ have size between $2^{-1/(2d)}\varepsilon m$ and $m$, 
we see that $|B_{-j}| \le 2^{1/(2d)}r\varepsilon^{-1}|B_a|$ for all $a \neq j$. 
Therefore, $|B_{-j}^{d}| \le 2^{1/2}(r\varepsilon^{-1})^d |\mathcal{B}|$, and
\[
	\mu_{s,\theta}(\mathcal{B}; B_j)
	= \frac{1}{|\mathcal{B}|} \sum_{Q \in \mathcal{B}} \omega_{\theta}(Q;B_j)^{s}
	\le \frac{2^{1/2}(r\varepsilon^{-1})^d}{|B_{-j}|^{d}} \sum_{Q \in B_{-j}^{d}}\omega_{\theta}(Q;B_j)^{s}
	= 2^{1/2}r^{d} \varepsilon^{-d} \mu_{s,\theta}(B_{-j}^d; B_j). 
\]
Thus, $E_2$ holds with probability greater than $\frac{1}{2}$.
We proved that there is non-zero
probability that all $E_1, E_2, E_3$ hold, and as seen above, this proves the lemma.
\end{proof}

\section{Concluding remarks} \label{sec:remarks}

\noindent \textbf{Original form of the Burr-Erd\H{o}s conjecture}.
The definition of Ramsey numbers can be extended to pairs of graphs. 
For a pair of graphs $H_1$ and $H_2$, the {\em Ramsey number} of the pair $(H_1, H_2)$,
denoted $r(H_1, H_2)$ is the minimum integer $n$ such that in every edge coloring of $K_n$
with two colors red and blue, there exists a red copy of $H_1$ or a blue copy of $H_2$. 
The {\em arboricity} of a graph is the minimum number of forests into which
its edge set can be partitioned. 
The original conjecture of Burr and Erd\H{o}s \cite{BuEr} can be stated as follows:

\begin{conj}
For every natural number $d$, there exists a constant $c$ such that for every pair
of graphs $H_1$ and $H_2$, each having arboricity at most $d$, 
$r(H_1, H_2) \le c(|V(H_1)| + |V(H_2)|)$.
\end{conj}

It is well-known that arboricity
and degeneracy are within a factor two of each other and hence Theorem~\ref{thm:main}
implies this conjecture. Moreover, our proof can be extended
to more than two colors.

\medskip

\noindent \textbf{Determining the constant}.
For graphs with fixed chromatic number, the constant $c_d$ 
we found is exponential in $d$, which is best possible up to the constant in the exponent. 
For general degenerate graphs, the constant $c_d$ we obtained is double-exponential in $d$, and
it still remains to understand the correct behavior of this constant. 
The corresponding question for bounded degree graphs is reasonably well-understood since 
Conlon, Fox, and Sudakov \cite{CoFoSu, Conlon, FoSu09}
proved $r(G) \le c^{\Delta \log \Delta} |V(G)|$ for all graphs of degree at most $\Delta$, 
and $r(G) \le c^{r\Delta} |V(G)|$ for all $r$-colorable graphs of degree at most $\Delta$.
These bounds are close to being best possible since Graham, R\"odl, and
Ruci\'nski proved that there are bipartite graphs of maximum degree $\Delta$ having
$r(G) \ge c^{\Delta}|V(G)|$ (for some different constant $c$). 
Moreover, the author \cite{Lee2} proved that a transference principle holds for bounded degree graphs,
and thus if $G$ has a `simple' structure, then 
the bound on the constant can be significantly improved. 
For example, if there exists a homomorphism $f$ from $G$ to a graph $H$ having 
maximum degree at most $d$ where $|f^{-1}(v)| = o(n)$ for each $v \in V(H)$, then 
$r(G) \le c^{d\log d}|V(G)|$. Hence, in this case the constant does not grow together 
with the maximum degree of $G$.

\medskip

\noindent \textbf{Further applications of the technique}.
We used a random greedy embedding algorithm together with
dependent random choice. Theorem~\ref{thm:success} shows that our embedding algorithm succeeds with
probability greater than $\frac{1}{2}$. 
Hence, a careful analysis will show that there are in fact many copies of the graph of interest. 
It would be interesting to find further applications of these methods. For instance, 
by using a variant
of the proof of Lemma~\ref{lem:drc_final} as in \cite{Lee} together with the embedding methods
developed in this paper, one can show that a weak version
of the blow-up lemma holds for degenerate graphs. Namely, for all $d, \delta$ and 
sufficiently large $n$, there exist
$\varepsilon$ and $c$ such that if $\{V_i\}_{i \in [r]}$ are
disjoint vertex subsets each having size at least $n$ and $(V_i, V_j)$ are $(\varepsilon,\delta)$-dense
for each distinct $i,j$, then it contains as subgraphs all $d$-degenerate $r$-colorable graphs
with at most $c^{d}n$ vertices.  
One might be able to further develop this approach as in \cite{Lee}, 
and extend the bandwidth theorem of B\"ottcher, Schacht, and Taraz \cite{BoScTa} to 
degenerate graphs of sublinear bandwidth.

\medskip

\noindent \textbf{Related problems}.
As observed by Burr and Erd\H{o}s, graphs with at least $(1+\varepsilon)n\log n$ edges 
have Ramsey numbers superlinear in the number of vertices. 
On the other hand, they showed that there are graphs with $cn \log n$ edges 
for some constant $c$ that have Ramsey numbers linear in the number of vertices.
It would be interesting to further classify the graphs that have Ramsey number linear
in terms of its number of vertices. An interesting test case is 
the class of hypercubes $Q_n$,
for which we slightly improved the previous best known bound to $r(Q_n) = (1 + o_n(1))2^{2n}$.
Burr and Erd\H{o}s conjectured that there exists a constant $c$ such that $r(Q_n) \le c 2^n$ holds 
for all natural numbers $n$.

In a similar direction, there has been much work aimed at understanding
the Ramsey number of a graph
in terms of its number of edges. The most notable result in this direction 
was proved by Sudakov \cite{Sudakov}, who confirmed a conjecture of Erd\H{o}s and
Graham by showing that $r(H) \le 2^{c\sqrt{m}}$ holds for all graphs $H$ with $m$ edges.
Also, Conlon, Fox, and Sudakov have an interesting conjecture 
\cite[Conjecture 2.16]{CoFoSu_survey}, asking whether $\log(r(H)) = \Theta(d(H) + \log n)$
holds for all $n$-vertex graphs $H$, where $d(H)$ is the degeneracy of $H$.

\medskip

\noindent \textbf{Acknowledgements}. I thank David Conlon, Jacob Fox, and Benny Sudakov for their valuable remarks. I also thank the anonymous referee for carefully reading the manuscript and making several useful comments.

\end{document}